\documentclass[11pt]{article}
\usepackage{amsfonts}
\usepackage{amsthm}
\usepackage{amsmath}
\usepackage{amssymb}

\usepackage{epsfig}
\usepackage[percent]{overpic}
\usepackage{graphics}
\usepackage{psfrag}

\usepackage{verbatim}
\usepackage{color}

\usepackage[all]{xy}
\usepackage{makeidx}
\usepackage{enumerate}
\usepackage[letterpaper,margin=0.9in]{geometry}

\usepackage{tikz}
\usetikzlibrary{matrix}

\let\oldbibliography\thebibliography
\renewcommand{\thebibliography}[1]{%
  \oldbibliography{#1}%
  \setlength{\itemsep}{-1.2mm}%
}

\theoremstyle{plain}
\newtheorem{thm}{Theorem}[section]
\newtheorem{cor}[thm]{Corollary}
\newtheorem{lem}[thm]{Lemma}
\newtheorem{prop}[thm]{Proposition}
\newtheorem{rem}[thm]{Remark}

\theoremstyle{definition}
\newtheorem{defn}[thm]{Definition}

\newtheorem{ex}[thm]{Example}

\newtheoremstyle{myremark}% name
  {3pt}%      Space above
  {3pt}%      Space below
  {\small \rmfamily}%         Body font
  {5pt}%         Indent amount (empty = no indent, \parindent = para indent)
  {\rmfamily}% Thm head font
  {:}%        Punctuation after thm head
  {.5em}%     Space after thm head: " " = normal interword space;
  {}%         Thm head spec (can be left empty, meaning `normal')

\theoremstyle{myremark}

\def\txtd{{\textnormal{d}}}
\def\txte{{\textnormal{e}}}

\def\ra{\rightarrow}
\def\I{\infty}

\newcommand{\be}{\begin{equation}}
\newcommand{\ee}{\end{equation}}
\newcommand{\benn}{\begin{equation*}}
\newcommand{\eenn}{\end{equation*}}
\newcommand{\bea}{\begin{eqnarray}}
\newcommand{\eea}{\end{eqnarray}}
\newcommand{\beann}{\begin{eqnarray*}}
\newcommand{\eeann}{\end{eqnarray*}}

\newcommand{\myendex}{$\blacklozenge$\end{ex}}
\newcommand{\myendexerc}{$\lozenge$\end{exerc}}
\newcommand{\myendpexerc}{$\lozenge$\end{pexerc}}

\begin{document}
\numberwithin{equation}{section}
\author{Marios-Antonios Gkogkas\thanks{
		Department of Mathematics, Technical University 
		of Munich, 85748 Garching b.~M\"unchen, Germany}~~and~~Christian Kuehn~\footnotemark[1]}

\title{Graphop Mean-Field Limits for Kuramoto-Type Models}

\maketitle

\begin{abstract}
Originally arising in the context of interacting particle systems in statistical physics, dynamical systems and differential equations on networks/graphs have permeated into a broad number of mathematical areas as well as into many applications. One central problem in the field is to find suitable approximations of the dynamics as the number of nodes/vertices tends to infinity, i.e., in the large graph limit. A cornerstone in this context are Vlasov-Fokker-Planck equations (VFPEs) describing a particle density on a mean-field level. For all-to-all coupled systems, it is quite classical to prove the rigorous approximation by VFPEs for many classes of particle systems. For dense graphs converging to graphon limits, one also knows that mean-field approximation holds for certain classes of models, e.g., for the Kuramoto model on graphs. Yet, the space of intermediate density and sparse graphs is clearly extremely relevant. Here we prove that the Kuramoto model can be be approximated in the mean-field limit by far more general graph limits than graphons. In particular, our contributions are as follows. (I) We show, how to introduce operator theory more abstractly into VFPEs by considering graphops. Graphops have recently been proposed as a unifying approach to graph limit theory, and here we show that they can be used for differential equations on graphs. (II) For the Kuramoto model on graphs we rigorously prove that there is a VFPE equation approximating it in the mean-field sense. (III) This mean-field VFPE involves a graphop, and we prove the existence, uniqueness, and continuous graphop-dependence of weak solutions. (IV) On a technical level, our results rely on designing a new suitable metric of graphop convergence and on employing Fourier analysis on compact abelian groups to approximate graphops using summability kernels.    
\end{abstract}


%%%%%%%%%%%%%%%%%%%%%%%%%%%%%%%%%%%%%%%%%%%%%%%%%%%%%%%%%%%%%%%%%%%%%%%%%%%%%%%%%%%%%%%%%%%%%%

\section{Introduction}
\label{sec:intro}

Synchronization, or in other words the effect under which a system of coupled oscillators with different individual initial frequencies pulses, after a while, under the same single global frequency, is a phenomenon which can be found in various biological, ecological, social and technological processes \cite{Strogatz2000}. An important first model for synchronization was developed by Kuramoto~\cite{Kuramoto75}. This model considers a finite number of $N$ different oscillators. Each oscillator has an intrinsic frequency $\omega_i \in \mathbb{R}$ for $i=1,...N$. The frequencies are distributed according to a symmetric probability density function $g:\mathbb{R} \to [0,\infty)$. The phase of each oscillator $u_i(t) \in[0,2\pi) =: \mathbb{T}$ are the unknowns satisfying the following system of ordinary differential equations (ODEs)
\begin{equation}
\label{eq: kuramotos model intro}
\dot{u}_i = \omega_i + \frac{C}{N}\sum_{j=1}^N\sin(u_j - u_i), \quad i \in [N]:= \{1,2,...,N\},
\end{equation}
where the parameter $C>0$ is the coupling strength. Further, let $\rho(u,\omega,t)~\txtd u$ denote the fraction of oscillators with frequency $\omega$ and phase between $u$ and $\txtd u$ for time $t$. Sakaguchi~\cite{Sakaguchi88} proposed that the Kuramoto model (\ref{eq: kuramotos model intro}) can be approximated, as $N\ra \I$, by the single mean field Vlasov-Fokker-Planck equation (VFPE)
\begin{equation}
\label{eq: mean field eq intro}
\frac{\partial{\rho}}{\partial{t} } + \frac{\partial{}}{\partial{u} }(\rho V(\rho)) = 0,
\end{equation}
with the characteristic field
\begin{equation}
V(\rho)(u,\omega,t) = \omega + C \int_0^{2\pi}\int_{\mathbb{R}} \sin(\tilde{u}-u)\rho(\tilde{u},\tilde{\omega},t)g(\tilde{\omega})~\txtd\tilde{\omega}~\txtd\tilde{u}.
\end{equation}
Although the formal derivation of (\ref{eq: mean field eq intro}) from (\ref{eq: kuramotos model intro}) has been extensively studied in the literature (see for example \cite{Crawford99,Strogatz2000} and references therein), it was only proved rigorously around fifteen years ago by Lancelloti \cite{Lancellotti2015}. His approach was to view (\ref{eq: mean field eq intro}) as an abstract continuity equation of measures and then apply Neunzert's fixed point argument \cite{Neunzert1978,Neunzert84}. Of course, the classical Kuramoto model~\eqref{eq: kuramotos model intro} makes the unrealistic assumption that every oscillator equally affects everyone else and that coupling takes place exactly via the first Fourier mode represented by the sine nonlinearity. For more precise models we should take into account the network coupling structure of the system and more general Fourier modes. This generalized Kuramoto-type model on an arbitrary network/graph takes the form
\begin{equation}
\label{eq: kuramotos model on graphs intro}
\dot{u}_i = \omega_i +\frac{C}{N}\sum_{j=1}^NA^N_{i,j}D(u_j - u_i), \quad i \in [N],
\end{equation}
where  $A = (A^N_{i,j})_{i,j =1,...N } \in \mathbb{R}^{N\times N}$ is the adjacency matrix of the network of oscillators and the coupling function $D:\mathbb{T} \to \mathbb{R}$ satisfies the Lipschitz condition 
\begin{equation}
\label{eq: Lipschitz condition for D}
| D(u) - D(v) | \leq | u - v |, \quad \forall u,v \in \mathbb{T}.
\end{equation}
Further, without loss of generality we may assume that
\begin{equation}
\label{eq: bound. for D}
\max_{u \in \mathbb{T}}|D(u)| \leq 1.
\end{equation}
The recent development of graph limit theory~\cite{Lovasz2006} enabled the rigorous treatment of approximating limits as $N\to \infty$ for several classes of graphs converging, in a suitable sense, towards a graph limit~\cite{Medv2014,Medv(sparse)2014,Medvedev2018,Medv2016}. For example, Kaliuzhnyi-Verbovetskyi and Medvedev \cite{Medvedev2018} treat the case that there exists a graphon limit $W:[0,1] \times [0,1] \to [0,1]$, i.e., $W$ is a measurable and symmetric function such that the weights $A^N_{i,j}$ are given by 
\begin{equation}
\label{eq: weights in intro}
A^N_{i,j}:= \int_{I^N_i\times I^N_j} W(x,y) ~\txtd x ~\txtd y.
\end{equation}
Here, $\{I^N_i\}_{i=1,...,N}$ is the partition of $I:= [0,1]$ given (up to measure 0) by the intervals $I^N_i:= [\frac{i-1}{N},\frac{i}{N}]$. Consider the family of empirical measures
\begin{equation}
\label{eq: empirical measure graphon intro}
\nu^x_{n,M,t}(S) = M^{-1} \sum_{j=1}^M \chi_S(u^N_{(i-1)M+j}(t)), \quad S \in \mathcal{B}(\mathbb{T}),\quad  x \in I^N_i, 
\end{equation}
where $\mathcal{B}(\mathbb{T})$ denotes the Lebesque $\sigma$-algebra on $\mathbb{T}$, $\chi_S$ is the indicator function of $S$, and $u^N_{i}(t)$ is the solution of (\ref{eq: kuramotos model on graphs intro}). Then one may prove~\cite{Medvedev2018} that the empirical measure~\eqref{eq: empirical measure graphon intro} approximates as $N\to \infty$, in a suitable distance and under certain initial conditions of the Kuramoto model (\ref{eq: kuramotos model on graphs intro}), the family of continuous measures
\begin{equation}
\nu_t^x(S) = \int_S\int_\mathbb{R} \rho(t,\tilde{u},\tilde{\omega},x)~\txtd\tilde{\omega}~\txtd\tilde{u},  \quad S \in \mathcal{B}(\mathbb{T}), \quad x \in I,
\end{equation} 
where $\rho$ solves the mean field equation
\begin{equation}
\label{eq: mean field eq graphons intro}
\frac{\partial{\rho}}{\partial{t} } + \frac{\partial{}}{\partial{u} }(\rho V[W](\rho)) = 0.
\end{equation}
As we might expect, this VFPE is similar to (\ref{eq: mean field eq intro}), with the only difference that the characteristic field $V=V[W]$ now depends on the graphon and is explicitly given by
\begin{equation}
V[W](\rho)(u,\omega,x,t) = \omega + C \int_{I}\int_0^{2\pi}\int_{\mathbb{R}}D(\tilde{u}-u) W(x,y)\rho(\tilde{u},\tilde{\omega},y,t) g(\tilde{\omega}) ~\txtd \tilde{\omega}~\txtd\tilde{u}~\txtd y.
\end{equation}
All recent approaches for the mean-field limit relied on the fact that the limiting graph is given by a graphon (i.e., it is a dense graph) and a natural question arises is, how to treat the case of limiting graphs with intermediate densities or sparse graphs? A major obstacle was that up until recently, infinite sparse and dense graphs had been extensively studied in the graph limit theory but using very different convergence notions, which relied frequently on combinatorial ideas difficult to incorporate into analysis-based methods used for differential equations. Even beyond this challenge, no unified theory existed for the treatment of graphs of intermediate density. Only recently Backhausz and Szegedy~\cite{Scededgy2018} provided a far more general framework unifying dense and sparse graph limit theory. The novel viewpoint is that graphs can be represented via suitable operators, the so-called graphops.

The main goal of the present paper is to prove that mean field approximation of the Kuramoto-type models of the form (\ref{eq: kuramotos model on graphs intro}) is possible in many cases if the limiting graph is given by a graphop. Since graphops cover a very large and general class of graph limits~\cite{Scededgy2018}, we believe that our results can also provide the basis for a very broad use of graphops in differential equations arising from dynamics on networks/graphs. Next, we are going to introduce some basic facts about graphops. After that, we come back to discuss the central question in this paper in more detail.

\subsection{Representation of graphs via graphops}

Graphops were introduced in \cite{Scededgy2018} as a new way for representing graphs to unify the language provided for dense graph limits and Benjamini-Schramm limits, and to include also graphs of intermediate density. We quickly summarize some basic notions and results given in \cite{Scededgy2018}. Let  $(\Omega, \Sigma,m)$ be a Borel probability space and $\Omega$ a compact set. A \textit{P-operator} is a linear operator $A:L^\infty(\Omega,m) \to L^1(\Omega,m)$ which is bounded, i.e., it has a finite operator norm
\begin{equation*}
\parallel A \parallel_{\infty \to 1} := \sup_{v \in L^\infty(\Omega)} \frac{\parallel Av \parallel_1}{\parallel v \parallel_\infty} < \infty.
\end{equation*}
More generally, for a $P$-operator $A$, the operator norm $\parallel A\parallel_{p\to q}$, for any real numbers $p,q\in [1,\infty]$, is given by
\begin{equation*}
\parallel A \parallel_{p\to q} := \sup_{v \in L^\infty(\Omega)} \frac{\parallel Av \parallel_q}{\parallel v \parallel_p}.
\end{equation*}
$\mathcal{PB}(\Omega,\Sigma,m)$ denotes the space of all P-operators on $\Omega$; if the underlying measure $m$ is clear we simply write $\mathcal{PB}(\Omega)$). In the space of P-operators, objects which represent graphs are the so-called graphops. A P-operator $A$ is called a \textit{graphop} if it is positivity preserving and self-adjoint. To be more precise, positivity preserving means that 
\begin{equation*}
v(x)\geq 0  \quad \text{for $m$-a.e.}~x\in \Omega \Rightarrow~Av(x) \geq 0 \quad \text{for $m$-a.e.}~x\in \Omega,
\end{equation*}
and self-adjoint here means that for any $v,w \in L^\infty(\Omega,m)$ we have
\begin{equation*}
\langle Av,w\rangle = \langle v,Aw\rangle,
\end{equation*}
with the bilinear from $\langle v,w \rangle := \int_{\Omega} v(x)w(x)~\txtd m(x)$. We also write $\langle v,w \rangle_A := \langle Av,w\rangle$.

Intuitively, the space $\Omega$ represents the \emph{node set} of the graph. Its \emph{edge set} is represented by a symmetric fiber measure $\nu$ on the product set $\Omega \times \Omega$, which exists for any graphop $A$ according to following theorem:

\begin{thm}\cite{Scededgy2018} \textbf{(Measure representation of graphops)} 
	\label{thm: measure repres of graphops}\\
	Assume that $A: L^\infty(\Omega,m) \to L^1(\Omega,m)$ is a graphop. Then following statements are true:
	\begin{enumerate}
		\item There is a unique finite measure $\nu$ on $(\Omega \times \Omega, \Sigma \times \Sigma)$ with the following properties: \\
		\textbf{(i)} $\nu$ is symmetric. \\
		\textbf{(ii)} The marginal distribution $\pi_*\nu$ of $\nu$ on $\Omega$ is absolutely continuous with respect to $m$. Here $\pi:\Omega\times \Omega \to \Omega$ denotes the canonical projection and $\pi_*$ is the associated pushforward.\\
		\textbf{(iii)} For every $f,g\in L^\infty(\Omega,m)$ holds:
		\begin{equation*}
		\langle f ,g \rangle _A = \int_{\Omega^2} f(x)g(y)~\txtd \nu(x,y) = \int_{\Omega^2} g(x)f(y)~\txtd \nu(x,y) = \langle g ,f \rangle _A .
		\end{equation*}	
		\item There is a family $\{\nu_x\}_{x \in \Omega}$ of finite measures (called \textit{fiber measures}), such that for all $f \in L^\infty(\Omega,m)$ we have 
		\begin{equation*}
		(Af)(x) = \int_{\Omega} f(y) ~\txtd\nu_x(y) \text{ \quad $m$-a.e. $x \in \Omega$.}
		\end{equation*}
		For this family we have additionally for any $h\in L^\infty(\Omega^2,\nu)$
		\begin{equation*}
		\int_{\Omega^2} h(x,y)~\txtd\nu_x(y)~\txtd m(x) = \int_{\Omega^2} h(x,y) ~\txtd \nu(x,y).
		\end{equation*}
	\end{enumerate}		
\end{thm} 

For a given graphop $A$,  we call the family $\{\nu_x\}_{x \in \Omega}$ the \textit{fiber measures associated to the graphop $A$}. We sometimes also write $\nu_x^A$ to make the graphop dependence clear. Notice that the second statement in Theorem \ref{thm: measure repres of graphops} follows immediately from the first one, using the disintegration theorem. For a node $x \in \Omega$, the measure $\nu_x$ represents the neighborhood of $x$. Moreover, the number $A\chi_\Omega(x) = \nu_x(\Omega)$ is the degree of $x$. A particular interesting case occur, when all edges have the same degree, or in other words, when there exists a constant $c>0$ such that $A\chi_\Omega(x) = c\chi_\Omega(x)$. In this case the graphop $A$ is called \textit{$c$-regular}. A \textit{Markov graphop} is a $1$-regular graphop. Another important case occurs when all fiber measures $\{\nu_x\}_{x \in \Omega}$ are absolutely continuous with respect to $m$. In this case, by the Radon-Nikodym theorem, we have 
$\nu_x(dy) = W(x,y)~\txtd m(y)$ for a symmetric function $W: \Omega \times \Omega \to [0,\infty)$. 

\begin{defn}\label{def: graphons}
	A \textit{graphon}\footnote{Different definitions are used by different authors. In the classical literature, a graphon is considered to be a function $W:I\times I \to I$, but it is also known that (under additional assumptions) we may identify $\Omega=I$ with the unit interval \cite[Theorem 7.1]{Janson10}} is a measurable (wrt the product $\sigma$-algebra), symmetric, bounded and positive function $W:\Omega\times \Omega \to [0,\infty)$. 
\end{defn} 

The other way around, given a graphon $W$ we may easily obtain a graphop $A_W: L^\infty(\Omega,m) \to L^1(\Omega,m)$, via setting 
\begin{equation}\label{eq: graphop corr to graphon}
A_Wf(x):= \int_\Omega W(x,y)f(y)~\txtd m(y).
\end{equation}
Thus we may view the graphon space as a true subspace of the space of graphops. In the following, we will often call $A_W$ given in (\ref{eq: graphop corr to graphon}) itself a graphon and $W$ the corresponding (graphon) kernel.

\subsection{Main problem}
\label{sec: main problem}

Our starting point is an observation made by the second author in~\cite{Kuehn2020}: If we compare the original VFPE (\ref{eq: mean field eq intro}) with the VFPE on a graphon (\ref{eq: mean field eq graphons intro}), we observe that in equation (\ref{eq: mean field eq intro}) we formally replace
\begin{equation*}
\rho(\tilde{u},\tilde{\omega},t) \text{\quad by \quad} \underbrace{\int_{I}W(x,y) \rho(\tilde{u},\tilde{\omega},y,t)~\txtd y}_{= A_W\rho(x; \tilde{u}, \tilde{\omega},t)}.
\end{equation*}
Therefore, in the case that the sequence of adjacency matrices $A^N$ for the discrete Kuramoto model (\ref{eq: kuramotos model on graphs intro}) is converging in the sense of dense graph convergence towards a graphon, the effect on the mean field VFPE is best viewed as an operator action. Having in mind this observation and linking it to the new operator framework for representing graphs via graphops, one may conjecture~\cite{Kuehn2020} formally that if the limiting object of the sequence of graphs, in the sense of P-operator convergence, is a general graphop $A$, then in 
equation (\ref{eq: mean field eq intro}) we should replace 
\begin{equation*}
\rho(\tilde{u},\tilde{\omega},t) \text{\quad by \quad} A\rho(x;\tilde{u},\tilde{\omega},t).
\end{equation*}
Our main goal in this paper is to prove this conjecture rigorously. Furthermore, we want to provide a suitable solution theory for VFPEs involving graphops. Let us quickly discuss the main idea, how we are going to prove the approximation properties of the mean-field VFPE based upon the results for graphons. 

From now on in this paper, for notational simplicity, we shall restrict to the case that all frequencies $\omega_i=0 , i \in[N]$ are identical zero, but all results can be extended in a straightforward way to the general case of unequal frequencies, see the discussion in Section \ref{sec:conclusion}. Let $A: L^\infty(\Omega,m) \to L^1(\Omega,m)$ be a fixed graphop with node set given by a compact abelian group $\Omega$ equipped with the Haar measure $\mu_\Omega$ on the Lebesque sets. Assume that for the graphop $A$ we have found a sequence of graphons $A^K$ with corresponding kernels $W^K: \Omega \times \Omega \to [0,\infty)$, such that
\begin{equation*}
A^K \to A \quad \text{ as } K \to \infty,
\end{equation*}
where the convergence takes place in a carefully chosen topology. Further let $(\Omega^n_i)_{i=1,...,n}$ be a sequence partitions of $\Omega$ satisfying $m(\Omega^n_i)=\frac{1}{n}$ for all $i\in[n]$. For any fixed $M\in \mathbb{N}$ we set $N=nM$ and assume additionally that the partition satisfies $\Omega^N_{(i-1)M + k} \subset \Omega^n_i$ for all $n\in \mathbb{N}$ and $i \in [n], k \in [M]$. We can then define, for any $N,K\in \mathbb{N}$,  the \textit{weights}
\begin{equation}
\label{eq: weights Kurmato on T}
A^{N,K}_{i,j}:= N^2\int_{\Omega^N_i\times \Omega^N_j} W^K(x,y) ~\txtd x ~\txtd y
\end{equation}
and consider the \textit{(generalized) Kuramoto model}
\begin{subequations}
	\label{eq: Kuramoto problem graphon approximable graphops}
	\begin{alignat}{4}
	\dot{u}^{N,K}_{i} &= C N^{-1} \sum_{j=1}^N A^{N,K}_{i,j} 
	D(u^{N,K}_{j} - u^{N,K}_{i}), \\
	u^{N,K}_{i}(0) &= u^{N,0}_{i}, i\in [N].
	\end{alignat}
\end{subequations}
We recall that the coupling function $D:\mathbb{T} \to \mathbb{R}$ satisfies conditions (\ref{eq: Lipschitz condition for D}) and (\ref{eq: bound. for D}). Since $A^K$ is for any fixed $K\in \mathbb{N}$ a graphon, we can conclude by \cite{Medvedev2018} (generalized for graphons defined on the space $\Omega$, cf. Remark \ref{rem: graphons on Omega}), that the empirical measure
\begin{equation}
\label{eq: empirical measure sec graphon approxm}
\nu^x_{n,M,K,t}(S) = M^{-1} \sum_{j=1}^M \chi_D(u^{N,K}_{(i-1)M+j}(t)), \text{\quad} x \in \Omega^n_i, \quad  S \in \mathcal{B}(\Omega)
\end{equation}
approximates, as $N\to \infty$, the $K$-th continuous measure
\begin{equation}
\label{eq: cont. measure rK sec graphon approxm}
\nu^{x,K}_t(S) = \int_S\rho^K(t,u,x)~\txtd u, \text{\quad} x \in \Omega, \quad S \in \mathcal{B}(\Omega).
\end{equation}
where $\rho^K$ is the unique solution of the following mean field initial value problem (IVP), denoted by \textit{VFPE$^K$}, 
\begin{subequations}
	\label{eq: IVP for the VFPE^K}
	\begin{alignat}{4}
	\partial_t\rho^K(t,u,x) &= - \partial_u(\rho^K V[A^K]\rho^K)(t,u,x), \quad (t,u,x)\in [0,T]\times \mathbb{T}\times \Omega, \\
	\rho^K(0,u,x) &= \rho^0(u,x),
	\end{alignat}
\end{subequations}
corresponding to the graphon $A^K$ with initial condition $\rho^0$. We further define the limiting measure for $K\ra \I$ by 
\begin{equation}
\label{eq: cont. measure r sec graphon approxm}
\nu^{x}_t(S) = \int_S\rho(t,u,x)~\txtd u, \text{\quad} x \in \Omega, \quad S \in \mathcal{B}(\Omega),
\end{equation} 
where $\rho$ is the unique solution (cf.~Theorem \ref{thm: existence and uniqness for weak solutions of the VFPE with graphops} below) of the limiting IVP, denoted by \textit{VFPE$^\infty$}, 
\begin{subequations}
	\label{eq: IVP for the VFPE^infty}
	\begin{alignat}{4}
	\partial_t\rho(t,u,x) &= - \partial_u(\rho V[A]\rho)(t,u,x), \quad (t,u,x)\in [0,T]\times \mathbb{T}\times \Omega, \\
	\rho(0,u,x) &= \rho^0(u,x),
	\end{alignat}
\end{subequations}
corresponding to the graphop $A$, with the characteristic field $V[A]$ be given by 
\begin{equation}
\label{eq: characteristic mean field V intro}
V[A]\rho(t,u,x) := C\int_0^{2\pi} (A\rho) (t,\tilde{u},x) D(\tilde{u} - u)~\txtd\tilde{u}.
\end{equation}
If the convergence of the graphon approximation $A^K$ towards $A$ is strong enough, we can then hope that, under suitable assumptions (which we will discuss in more detail later), the $K$-th continuous measure (\ref{eq: cont. measure rK sec graphon approxm}) will be close to the measure (\ref{eq: cont. measure r sec graphon approxm}). In particular, the following diagram summarizes the proof technique: 
\begin{center}
	\begin{tikzpicture}
	\matrix (m) [matrix of math nodes,row sep=3em,column sep=4em,minimum width=2em]
	{		
		\text{Kuramoto's model (\ref{eq: Kuramoto problem graphon approximable graphops})} & \text{VFPE$^K$ (\ref{eq: IVP for the VFPE^K})}  \\
		&\text{VFPE$^\infty$ (\ref{eq: IVP for the VFPE^infty})} \\};
	%	\text{Kuramoto problem (\ref{2.1}}) &  \text{VFPE$^\infty$}\\};
	\path[-stealth]
	%	(m-1-1) edge[dashed,-] node [left] {$\approx$} (m-2-1)
	(m-1-1)	edge  node [below] {$N \to \infty $} (m-1-2)
	%(m-1-2)edge  node [below] {$K \to \infty $} (m-1-3)
	%	(m-2-1.east|-m-2-2) edge node [below] {$N \to \infty $}
	%	node [above] {} (m-2-2) 
	(m-1-2)edge node [right] {$K \to \infty$} (m-2-2)
	(m-1-1)edge[dashed] node [left] {$N,K\to \infty $} (m-2-2);
	%	(m-1-2) edge node [right] {$n \to \infty$} (m-2-2);
	%edge [dashed,-] (m-2-1);
	\end{tikzpicture}
\end{center}

The basic advantage of this approach is that, once we passed to the first limit $N\to \infty$, we can forget the discrete Kuramoto model (\ref{eq: Kuramoto problem graphon approximable graphops}) and we only have to work with a VFPE. A central point for this approach to succeed is that the  convergence of the approximating sequence $A^K$ towards $A$ should be 
\begin{itemize}
	\item\textit{weak enough}, to allow approximation via graphons of a big enough class of graphops.
	\item \textit{strong enough}, to guarantee that solutions of the VFPE for different graphops, which are ``close enough'' with respect to this topology, are themselves arbitrary close.
\end{itemize}
In particular, this is the analytic translation of the key problem in graph limit theory on the level of VFPEs. In our case, the following new convergence notion will actually work:
\begin{defn} 
	\textbf{(o-graphop convergence)} \label{defn: o-graphop conv} \\
	For the graphops $A^n, A$ on the same probability space $(\Omega,\Sigma,m)$ with associated fiber measures $\nu_x^n$ and $\nu_x$ we write
	\begin{equation*}
	A^n \to_o A  \quad :\Leftrightarrow \quad  \nu_x^n \to_w \nu_x \quad m-a.e.~x\in \Omega.
	\end{equation*}
\end{defn}

\begin{rem}
	\label{rem: pointwise convergence and graphop convergence}
	Assume that $\Omega$ is also a metric space.
	By  the definition of weak convergence \cite[Definition 13.12]{KlenkeBook2014} and  Portmanteau's Theorem \cite[Theorem 13.16]{KlenkeBook2014}, it follows immediately that o-convergence is equivalent to any of the following conditions: \\ 
\textbf{(i)	}$A^nf(x) \to Af(x)$ m-a.e.~$x\in \Omega$ for all
 continuous functions $f: \Omega\to \mathbb{R}$. \\
\textbf{(ii)	}$A^nf(x) \to Af(x)$ m-a.e.~$x\in \Omega$ for all
Lipschitz continuous functions $f: \Omega\to \mathbb{R}$. \\
\textbf{(iii)}$A^nf(x) \to Af(x)$ m-a.e.~$x\in \Omega$
	for all bounded functions $f: \Omega\to \mathbb{R}$ with $m(U_f)=0$, where $U_f$ denotes the set of points of discontinuity of $f$.  
\end{rem}

\subsection{Outline of the paper}

In Section \ref{sec:existence}, we build up necessary results so that our main question concerning mean field approximation for the Kuramoto model (\ref{eq: Kuramoto problem graphon approximable graphops}) is well posed. In particular, we start by studying the general VFPE (\ref{eq: IVP for the VFPE^infty}) with a graphop $A$, defined on an arbitrary compact Borel probability space $(\Omega,\Sigma,m)$. For this equation, we prove existence and uniqueness of solutions, cf.~Theorem \ref{thm: existence and uniqness of solutions for the fixed point eq} and Theorem \ref{thm: existence and uniqness for weak solutions of the VFPE with graphops}. Then, in Section \ref{sec: cont depen of the solution of the VFPE on the graphop} we prove that the solutions of the VFPE depend continuously on the graphop, cf.~Proposition \ref{prop: continuous dependence on the graphop II}. After this, in Section \ref{sec: graphon approx of graphops on compact groups}, using tools from classical Fourier Analysis, we show that in the case that the node space $\Omega$ is a compact abelian group which is equipped with the Haar measure $\mu_\Omega$, any graphop $A$ can be approximated (in the sense of Definition \ref{defn: o-graphop conv}) by suitable graphon regularizations $A^K$, cf.~Proposition \ref{prop: o graphon approx using summabi kernels}. These regularizations are obtained via convolution with summability kernels. Using these approximations, we finally come back in Section \ref{sec: mean field approximation} to our main question, which was summarized in the previous Section~\ref{sec: main problem}, i.e., to prove mean field approximation of the discrete Kuramoto model (\ref{eq: Kuramoto problem graphon approximable graphops}) via the VFPE (\ref{eq: IVP for the VFPE^infty}), cf.~Theorem \ref{thm: vfpe approximates the discrete Kuramoto's problem for graphon approximable graphops}. We prove the mean field approximation for a big class of graphops on compact abelian groups, which covers the case that $A$ is any $c$-regular or Markov graphop, cf.~Corollary \ref{cor: Mean field approximation for c-regular graphops}. Finally, in Section \ref{sec: conclusion and outlook} we summarize our results and discuss further generalizations and open problems. We also discuss the straightforward adaptations needed for the treatment of general initial frequencies $\{\omega_i\}_{i \in \mathbb{N}}$.

\section{Existence and uniqueness of solutions for the VFPE}
\label{sec:existence}

Assume $(\Omega,\Sigma,m)$ is a compact Borel probability space and that $A\in \mathcal{PB}(\Omega,\Sigma,m)$ is a fixed graphop with corresponding measure $\nu= \nu^A$ on $\Omega^2$, cf.~Theorem~\ref{thm: measure repres of graphops}. Further, let $\{\nu^A_x\}_{x \in \Omega}$ be the family of fiber measures associated to the graphop A, i.e., for all $f \in L^\infty(\Omega,m)$ we have:
\begin{equation}
(Af)(x) = \int_{\Omega} f(y) ~\txtd\nu_x^A(y), \quad  m\text{-a.e. } x\in \Omega,
\end{equation}
cf.~Theorem \ref{thm: measure repres of graphops}. Since we work with the fixed given graphop $A$, we can (and will) always assume in this section that $\{\nu_x^A\}_{x\in \Omega}$ are fixed representatives defined on the whole space $\Omega$. In this section we are interested in proving existence and uniqueness of weak solutions for the general VFPE (\ref{eq: IVP for the VFPE^infty}) with the graphop $A$. We consider following initial value problem (IVP) for the VFPE 
\begin{subequations}
	\label{eq: IVP for the VFPE}
	\begin{alignat}{4}
	\partial_t\rho(t,u,x) &= - \partial_u(\rho V(\rho))(t,u,x) \quad (t,u,x)\in [0,T]\times \mathbb{T}\times \Omega, \\
	\rho(0,u,x) &= \rho^0(u,x),
	\end{alignat}
\end{subequations}
where the characteristic field $V(t,u,x) = V[A,\rho,x](t,u)$ is given by
\begin{equation}
\label{eq: characteristic mean field V}
V[A,\rho,x](t,u) := C\int_0^{2\pi} (A\rho) (t,\tilde{u},x) D(\tilde{u} - u) ~\txtd\tilde{u}.
\end{equation}
For any positive time $T> 0$ we set $\mathcal{T}:=[0,T]$.	Following \cite{Neunzert84,Medvedev2018}, we state following definition:

\begin{defn} \textbf{(Weak solutions for the VFPE)}\\
	A measurable function $\rho: \mathcal{T} \times \mathbb{T} \times \Omega \to \mathbb{R}$ is called a \textit{ weak solution of the IVP (\ref{eq: IVP for the VFPE}) for the VFPE}, if following conditions hold for $m$-a.e. $x\in \Omega$: 
	\begin{enumerate}
		\item $\rho(t,u,x)$ is weakly continuous in $t\in \mathbb{T}$, i.e., the map $t\mapsto \int_{\mathbb{T}} \rho(t,u,x) f(u) ~\txtd u$ is continuous for every $f\in C(\mathbb{T})$;
		\item for every $w\in C^1(\mathcal{T}\times \mathbb{T})$ with support in $[0,T) \times \mathbb{T}$ it holds that
		\begin{equation*}	   
		\int_0^T \int_{\mathbb{T}} \rho(t,u,x)\Big( \partial_t w(t,u) + V(t,u,x)\partial_u w(t,u)   \Big) ~\txtd u 	~\txtd t+ \int_{\mathbb{T}} w(0,u) \rho_0(u,x) ~\txtd u =0.
		\end{equation*}	
	\end{enumerate}	
\end{defn}

It can be shown \cite[remarks after eq. (10)]{Neunzert84} that if $\rho$ and $V$ are both sufficiently smooth, then $\rho$ is also a classical solution of the IVP of the VFPE (\ref{eq: IVP for the VFPE}). As we are going to see later in Section \ref{sec: the fixed point eq}, Neunzert's fixed point argument \cite{Neunzert1978,Neunzert84} translates the VFPE (\ref{eq: IVP for the VFPE}) to a fixed point equation for measures. Hence, let us now define the measure spaces we are going to work with.

\subsection{The measure spaces}\label{sec: the measures spaces}

Let $\mathcal{M}_f = \mathcal{M}_f(\mathbb{T})$ denote the space of finite Borel measures equipped with the bounded Lipschitz metric
\begin{equation*}
d_{BL}(\mu, \nu) :=  \sup_{f\in \mathcal{L}}\Big| \int_\mathbb{T} f(v) ~\txtd(\mu-\nu) (v) \Big|
\end{equation*}
where 
\begin{equation*}
\mathcal{L}:= \{ f:\mathbb{T} \to [0,1], \text{f is Lipschitz with Lipschitz constant $\leq$ 1} \}.
\end{equation*}
It is known that $(\mathcal{M}_f, d_{BL})$ is a complete metric space. Further, for any $b> 0$, we define the space
\begin{equation*}
\bar{\mathcal{M}}^{b}:= \{ \bar{\mu}: \Omega \to \mathcal{M}_f(\mathbb{T}):  \bar{\mu} \text{ is measurable, and } \sup_{x \in \Omega} \mu^x(\mathbb{T}) \leq b \},
\end{equation*}
where $\mu^x\in \mathcal{M}_f(\mathbb{T})$ denotes the evaluation of the family of measures $\bar{\mu}$ at $x$. In the following we assume additionally that for the graphop $A$ the following condition is satisfied 
\begin{equation}
\label{eq: cond: gamma_A leq 1}
\gamma_A:= \sup_{x \in \Omega}\Big(\nu_x^A(\Omega) \Big) \leq 1.
\end{equation}
Then, on $\bar{\mathcal{M}}^b\times\bar{\mathcal{M}}^b $ and with the graphop $A$, we set
\begin{equation*}
\bar{d}^{b,A} (\bar{\mu},\bar{\kappa} ):=  \sup_{x\in \Omega} \underbrace{\Big(\int_\Omega d_{BL}(\mu^y,\kappa^y) ~\txtd\nu_x^A(y) \Big)}_{=: \bar{d}^{b,A,x}}.
\end{equation*}
We further define the sets
\begin{align*}
\mathcal{G} &:= \{ B \in \mathcal{PB}(\Omega,\Sigma,l): l \text{ is a probability measure}, B \text{ is a graphop, } \gamma_B \leq 1\}, \\
\mathcal{G}^{m} &:= \{ B \in \mathcal{PB}(\Omega,\Sigma,m): B \text{ is a graphop }, \gamma_B \leq 1\},
\end{align*}
and the following metric (cf.~Lemma \ref{lem: db and db2 are equival metrics}) on $\bar{\mathcal{M}}^b$ 
\begin{equation}
\bar{d}^{b} (\bar{\mu},\bar{\kappa} ):= \sup_{B \in \mathcal{G}} \bar{d}^{b,B} (\bar{\mu},\bar{\kappa} ) = \sup_{B \in \mathcal{G}} \sup_{x\in \Omega} \Big(\int_\Omega d_{BL}(\mu^y,\kappa^y) ~\txtd\nu_x^B(y) \Big)
\end{equation}
It can be shown that $(\bar{\mathcal{M}}^{b},\bar{d}^{b})$ is a complete metric space, cf.~Lemma  \ref{lem: Mb is complete}. We also define the space
\begin{equation*}
\mathcal{M}_\mathcal{T}^{b} := C(\mathcal{T}, (\bar{\mathcal{M}}^b, \bar{d}^{b}))
\end{equation*}
and equip it with following metric for a fixed $\alpha> 0$ 
\begin{equation}
d_\alpha^{b}(\bar{\mu}_\cdot,\bar{\nu}_\cdot):= \sup_{ t\in \mathcal{T}} \txte^{-\alpha t} \bar{d}^{b}(\bar{\mu}_t,\bar{\nu}_t).
\end{equation}
We note that
\begin{equation*}
\txte^{-\alpha T}d^{b}(\bar{\mu}_\cdot,\bar{\nu}_\cdot) \leq d_\alpha^{b}(\bar{\mu}_\cdot,\bar{\nu}_\cdot) \leq d^{b}(\bar{\mu}_\cdot,\bar{\nu}_\cdot),
\end{equation*} 
where the metric
\begin{equation*}
d^{b}(\bar{\mu}_\cdot,\bar{\nu}_\cdot):=  \sup_{ t\in \mathcal{T}}  \bar{d}^{b}(\bar{\mu}_t,\bar{\nu}_t)
\end{equation*}
generates the usual uniform topology on $\mathcal{M}_\mathcal{T}^{b}$. Hence, from Lemma \ref{lem: Mb is complete} it follows that the space $(\mathcal{M}_\mathcal{T}^{b}, d_\alpha^{b})$ is complete as well.

\subsection{The extended graphop} \label{sec: the extended graphop}  

Associated with the graphop $A$ we can define an operator $\mathcal{A}$ on a %Bochner integrable 
family   $\{\mu^y\}_{y \in \Omega} \in \bar{\mathcal{M}}^b$ via
\begin{equation}\label{eq: extention of the measn field op on measures}
(\mathcal{A} \mu)^x := \int_{\Omega} \mu^y ~\txtd\nu_x^A(y).
\end{equation}
Here, the integral in the right side is to be understood in the following sense
\begin{equation*}
(\mathcal{A} \mu)^x (S) =\int_\Omega \mu^y(S) ~\txtd\nu_x^A(y) \text{\quad for any Borel set } S\subset\mathbb{T} \text{ and } x\in \Omega.
\end{equation*}
Note especially that for the given fixed family $\{\nu_x^A\}_{x \in \Omega}$,  $\{ (\mathcal{A} \mu)^x \}_{x \in \Omega}$, is a family of finite measures with
\begin{equation} \label{eq: Am leq b}
(\mathcal{A} \mu)^x (\mathbb{T}) \leq  \nu_x^A(\Omega)\sup_{y\in \Omega}\mu^y(\mathbb{T}) \leq b\gamma_A \leq b.
\end{equation}
Note that the operator $\mathcal{A}$ depends directly on fiber measures $\{\nu_x^A\}_{x\in \Omega}$. We will very often make use of the following lemma:

\begin{lem}\textbf{ ($\mathcal{A}$ and integration)}\label{lem: comput of dAmu}\\
	For any nonnegative Borel measurable function $f:\mathbb{T} \to \mathbb{R}_{\geq 0}$ and any family $\{\mu^y\}_{y \in \Omega} \in \bar{\mathcal{M}}^b$ we have
	\begin{align*}
	\int_{\mathbb{T}} f(v) ~\txtd(\mathcal{A}\mu^x)(v) &=\int_{\Omega}\int_\mathbb{T} f(v) ~\txtd\mu^y(v)~\txtd\nu_x^A(y) \\
	&= A\Big(\int_\mathbb{T} f(v) ~\txtd\mu^{\cdot}(v)\Big)(x)
	\end{align*}
	We also write in short notation
	\begin{equation*}
	\txtd(\mathcal{A}\mu^x)(v) =  ~\txtd\mu^y(v)~\txtd\nu_x^A(y).
	\end{equation*}
\end{lem}

\begin{proof}
	For the special case that $f =\chi_S$, where $S\subset{T}$ is Borel set, is a characteristic function we calculate immediately
	\begin{align*}
	\int_{\mathbb{T}} f(v) d(\mathcal{A}\mu^x)(v) &= \int_{S}  d(\mathcal{A}\mu^x)(v) = \Big(\int_\Omega \mu^y ~\txtd\nu_x^A(y) \Big)(S) \\
	&= \int_\Omega \mu^y(S) ~\txtd\nu_x^A(y) = \int_\Omega \int_{\mathbb{T}} f(v) ~\txtd\mu^y(v) ~\txtd\nu_x^A(y) \\
	&= A \Big(\int_\mathbb{T} f(v) ~\txtd\mu^{\cdot}(v) \Big)(x).
	\end{align*}
	In the same way, using linearity, we can verify the claim in the case that $f= \sum_{k=1}^n \alpha_k \chi_{S_k}$ is a simple function. For a general $f$ we approximate it by simple functions.	
\end{proof}

\begin{rem} \textbf{($\mathcal{A}$ is the canonical extention of the graphop $A$)} \\
	Using the previous lemma and Fubini's theorem, it is easy to check that if $\bar{\mu} \in \bar{\mathcal{M}}^b$ is absolutely continuous with respect to the Lebesque measure  with Radon-Nikodym derivative $\rho$, i.e., $~\txtd\mu^x(u) = \rho(u,x) ~\txtd u$, then $\mathcal{A} \mu^x$ is also absolutely continuous with $\txtd(\mathcal{A}\mu)^x(u) = (A\rho(u,\cdot))(x) ~\txtd u$. This implies that the definition of $\mathcal{A}$ we have provided is the correct extension of the graphop $A$ to measures to work with. 	
\end{rem}

\begin{lem}
	\label{lem: estim dx bounds dWA}
	For any $\bar{\mu}, \bar{\kappa} \in \bar{\mathcal{M}}^b$ and $x\in \Omega$ we have 
	\begin{equation*}
	d_{BL}(\mathcal{A}\mu^x, \mathcal{A}\kappa^x) \leq  \bar{d}^{b,A,x}(\bar{\mu},\bar{\kappa}) 	
	\leq \bar{d}^{b,A}(\bar{\mu},\bar{\kappa})  \leq \bar{d}^{b}(\bar{\mu},\bar{\kappa}) 
	\end{equation*}
\end{lem}

\begin{proof}
	This follows immediately from Lemma \ref{lem: comput of dAmu} and the definition of $d_{BL}$, since 
	\begin{align*}
	d_{BL}(\mathcal{A}\mu^x, \mathcal{A}\kappa^x) &=
	\sup_{f\in \mathcal{L}}\Big|\int_\Omega \int_\mathbb{T} f(v) ~\txtd(\mu^y- \kappa^y)(v) ~\txtd\nu_x^A(y)\Big| \\&\leq
	\int_\Omega \underbrace{ \sup_{f\in \mathcal{L}}\Big|\int_\mathbb{T} f(v) ~\txtd(\mu^y- \kappa^y)(v)\Big|}_{= d_{BL}(\mu^y,\kappa^y)} ~\txtd\nu_x^A(y) \\
	&\leq 	\int_\Omega  d_{BL}(\mu^y,\kappa^y)~\txtd\nu_x^A(y) = \bar{d}^{b,A,x}(\bar{\mu},\bar{\kappa})\\
	&\leq \sup_{x \in \Omega} \Big(\int_\Omega d_{BL}(\mu^y,\kappa^y)~\txtd\nu_x^A(y)\Big) = \bar{d}^{b,A}(\bar{\mu},\bar{\kappa})\\ 
	&\leq\bar{d}^{b}(\bar{\mu},\bar{\kappa}).
	\end{align*}
	This finishes the proof.
\end{proof}

\begin{lem} \textbf{(Lipschitz continuity of $\mathcal{A}$)} 
	\label{lem: Lips cont of A} \\
	The map $\mathcal{A} : \bar{\mathcal{M}}^b \to \bar{\mathcal{M}}^b$ is well-defined. Further, for every $\bar{\mu}, \bar{\kappa} \in \bar{\mathcal{M}}^b$ holds
	\begin{equation*}
	\bar{d}^{b}(\mathcal{A}\bar{\mu}, \mathcal{A}\bar{\kappa}) \leq 	\bar{d}^{b}(\bar{\mu}, \bar{\kappa}).
	\end{equation*}
\end{lem}

\begin{proof}
	Due to equation (\ref{eq: Am leq b}) it is easy to see that  $\mathcal{A}$ maps  $\bar{\mathcal{M}}^b$ to  $\bar{\mathcal{M}}^b$.	For the second statement, note that by Lemma \ref{lem: estim dx bounds dWA} we know that 
	\begin{align*}
	d_{BL}(\mathcal{A}\mu^x, \mathcal{A}\kappa^x) 
	&\leq \bar{d}^{b}(\bar{\mu},\bar{\kappa}).
	\end{align*}
	Hence the claim follows by integrating over $\nu^B_y$ and taking the  supremum over all $y\in \Omega$ and $B \in \mathcal{G}$.
\end{proof}

\subsection{The extended characteristic field} \label{sec: the extended char field}

Via the map $\mathcal{A}$, we can now extend the mean field vector field $V$, defined in (\ref{eq: characteristic mean field V}), from densities to measures, by defining
\begin{equation}\label{eq: mean field op on measures}
V[\mathcal{A}, \mu, x] (t,u) := C \int_0^{2\pi} D(\tilde{u} - u) ~\txtd(\mathcal{A}\mu_t)^x(\tilde{u})
\end{equation}
for any $x\in \Omega,t\in \mathcal{T},u \in \mathbb{T}$ and $\mu \in \mathcal{M}^{b}_{\mathcal{T}}$. 

\begin{lem} \textbf{(Regularity of the characteristic field $V$)} 
	\label{lem: regularity of V} \\
	The following statements are true:\\	
	\textbf{(I)} $V$ satisfies a Lipschitz condition in $\mu$ in the sense that for all $x \in \Omega$, $u\in \mathbb{T}$, $t\in \mathcal{T}$ and any $\mu ,\kappa \in \mathcal{M}^{b}_\mathcal{T}$ we have
	\begin{equation*}
	|V[\mathcal{A}, \mu,x](t,u) -  V[\mathcal{A}, \kappa,x](t,u)| \leq 2Cd^{b,A,x}(\bar{\mu}_t,\bar{\kappa}_t )\leq 2C\bar{d}^{b}(\bar{\mu}_t,\bar{\kappa}_t)
	\end{equation*}
	\textbf{(II)} For any $\mu \in \mathcal{M}^{b}_\mathcal{T}$ and $x \in \Omega$, the map $V[\mathcal{A},\mu,x](\cdot)$ is continuous in $(t,u)$ and Lipschitz continuous in $u$  uniformly in $t$ with Lipschitz constant bounded by $b\gamma_A$.
\end{lem}

\begin{proof} \textbf{(I)} We compute that 
	\begin{align*}
	|V[\mathcal{A}, \mu,x](t,u) -  V[\mathcal{A}, \kappa,x](t,u)| &=   C\Big| \int_0^{2\pi} D(\tilde{u} - u) ~\txtd(\mathcal{A}(\mu_t)^x -\mathcal{A}(\kappa_t))^x(\tilde{u}) \Big|\\
	&\leq  C \Big( \Big|\int_\mathbb{T} \chi_{I_{\geq 0}(u)}D(\tilde{u}-u) ~\txtd(\mathcal{A}(\mu_t)^x - \mathcal{A}(\kappa_t)^x)(\tilde{u}) \Big| \\&\quad \quad  + \Big|\int_\mathbb{T} -\chi_{I_{\leq 0}(u)} D(\tilde{u} - u) ~\txtd(\mathcal{A}(\mu_t)^x - \mathcal{A}(\kappa_t)^x)(\tilde{u}) \Big|  \Big),
	\end{align*}
	where $I_{\geq 0}(u) \subset \mathbb{T}$ is the interval where $D(\tilde{u} - u)$, as a function of $\tilde{u}$, is positive and $I_{\leq 0}(u)$ is the complementary set. Note that the functions $f^u_{+}(\tilde{u}):= \chi_{I_{\geq 0}(u)}D(\tilde{u} - u)$ and  $f^u_{-}(\tilde{u}):= -\chi_{I_{\leq 0}(u)}D(\tilde{u} - u)$ both lie in the set $\mathcal{L}$, so that we obtain, continuing the previous calculation and using Lemma~\ref{lem: estim dx bounds dWA}
	\begin{align*}
	|V[\mathcal{A}, \mu,x](t,u) -  V[\mathcal{A}, \kappa,x](t,u)| &\leq 
	2C 	d_{BL}(\mathcal{A}\mu^x_t, \mathcal{A}\kappa^x_t) \leq 2C \bar{d}^{b,A,x}(\bar{\mu}_t,\bar{\kappa}_t) \leq 2C \bar{d}^{b}(\bar{\mu}_t,\bar{\kappa}_t).
	\end{align*}
	\textbf{(II)} The proof proceeds in three steps. First, we are going to show Lipschitz continuity in $u$, then continuity in $t$ and finally continuity in $(t,u)$. We start with Lipschitz continuity in $u$. For any $u$ and $u_0 \in \mathbb{T}$ we have
	\begin{align*}
	|V[\mathcal{A},\mu,x](t,u) - V[\mathcal{A},\mu,x](t,u_0)| &= C\int_0^{2\pi} \Big|D(\tilde{u} - u) - D(\tilde{u} - u_0)\Big| ~\txtd(\mathcal{A}\mu_t)^x(\tilde{u}) \\
	&\leq  \underbrace{(\mathcal{A}\mu_t)^x(\mathbb{T})}_{\leq \nu_x^A(\Omega) \sup_{y \in \mathbb{T}  } \mu_y(\mathbb{T})\leq b\gamma_A} | u - u_0|. 
	\end{align*} 
	Next, we consider	continuity in $t$. For any $u \in \mathbb{T}$ and $t_0 \in \mathcal{T}$ we have for $t \to t_0$, using Lemma \ref{lem: estim dx bounds dWA} and the fact that $\mu \in \mathcal{M}^{b}_\mathcal{T}$ and with a similar calculation as in \textbf{(I)}:
	\begin{align*}
	|V[\mathcal{A},\mu,x](t,u) - V[\mathcal{A},\mu,x](t_0,u)| &\leq 2C d_{BL}(\mathcal{A}(\mu_t)^x,\mathcal{A}(\mu_{t_0})^x) \\
	&\leq 2C \bar{d}^{b,A,x}(\bar{\mu}_t,\bar{\mu}_{t_0}) \to 0  \text{\quad as } t \to t_0.
	\end{align*} 
	It remains to show continuity in $(t,u)$. For $(u,t), (u_0,t_0) \in \mathbb{T}\times \mathcal{T}$ we have 
	\begin{align*}
	V[\mathcal{A},\mu,x](t,u) - V[\mathcal{A},\mu,x](t_0,u_0)| &\leq 	V[\mathcal{A},\mu,x](t,u) - V[\mathcal{A},\mu,x](t,u_0)| \\
	&+ 	V[\mathcal{A},\mu,x](t,u_0) - V[\mathcal{A},\mu,x](t_0,u_0)|.
	\end{align*}	
	The second difference goes to 0 as $t \to t_0$ due to continuity in $t$. For the first difference we have
	\begin{align*}
	V[\mathcal{A},\mu,x](t,u) - V[\mathcal{A},\mu,x](t,u_0)| &= C| \int_0^{2\pi} \Big(D( \tilde{u} - u) - D(\tilde{u} - u_0) \Big)~\txtd(\mathcal{A}\mu_t)^x(\tilde{u}) | \\
	&\leq C | u - u_0| \nu_x^A(\Omega)\sup_{y \in \mathbb{T}} \mu_y(\mathbb{T}) \to 0 \text{\quad as } (t,u) \to (t_0,u_0).
	\end{align*}
	The claim follows.	
\end{proof}

\subsection{The equation of characteristics and the fixed point equation}
\label{sec: the fixed point eq}

For an arbitary $\mu \in \mathcal{M}^{b}_{\mathcal{T}}$ and for any $x\in \Omega$ we define following \textit{equation of characteristics} for a point $P= u \in \hat{G}:= \mathbb{T}$:
\begin{subequations}\label{eq: characteristic equation for the mean field 2}
	\begin{alignat}{4}
	\frac{dP}{dt} &= V[\mathcal{A},\mu,x](t,P), \\
	P(t_0) &= P_0.
	\end{alignat}	
\end{subequations}
Note that, due to Lemma \ref{lem: regularity of V}  we have that  equation (\ref{eq: characteristic equation for the mean field 2}) generates the flow 
\begin{equation}
T_{t,t_0}[\mathcal{A},\mu,x]: \hat{G} \to \hat{G},  \text{\quad} P_0 \mapsto P(t), P(t) \text{ solves (\ref{eq: characteristic equation for the mean field 2}) }.
\end{equation}
Note further that, if $D$ is smooth, then the regularity of $V[\mathcal{A},\mu,x](\cdot)$ in $u$, implies that $T_{t,t_0}$ is a $C^\infty$ diffeomorphism, satisfying $T_{t,t_0}^{-1} =T_{t_0,t}$.

\begin{defn}
	We say that a measure $\kappa \in \mathcal{M}^{b}_{\mathcal{T}}$ satisfies the \textit{fixed point equation asociated with the VFPE (\ref{eq: IVP for the VFPE})} with initial condition $\bar{\mu}_0 \in \bar{\mathcal{M}}^{b}$, if $\kappa$ satisfies 
	\begin{equation}\label{eq: fixed point eq assoc with the VFPE}
	\kappa_t^y = \mu_0^y \circ T_{0,t}[\mathcal{A},\kappa,y],  \text{ \quad for all $y \in \Omega$ }.
	\end{equation}%	\quad \nu_x-a.e. y \in \Omega
\end{defn}

\begin{lem} \textbf{(Properties of the characteristic flow)} \label{lem: V, T are uniformly bounded}\\
	The following statements are true: \\
	\textbf{(i)} $V[\mathcal{A},\cdot](\cdot)$ is uniformly bounded (in $x\in \Omega,t\in \mathcal{T},u \in \mathbb{T}$ and $\mu \in \mathcal{M}^{b}_{\mathcal{T}}$). \\
	\textbf{(ii)} The corresponding flow $T^x_{t,t_0}[\mu]u$ is uniformly bounded (in $x\in \Omega,t,t_0\in \mathcal{T},u \in \mathbb{T}$ and $\mu \in \mathcal{M}^{b}_{\mathcal{T}}$). \\
	\textbf{(iii)} $T^x_{t,t_0}[\mu]$ is Lipschitz continuous with Lipschitz constant $\txte^{Tb\gamma_A}$. 
\end{lem}

\begin{proof} 
	\textbf{(i)} We calculate using equation (\ref{eq: Am leq b})
	\begin{align*}
	\Big|V[\mathcal{A},\mu, x](t,u)\Big| &\leq 
	C\int_0^{2\pi}  \underbrace{\Big|D(\tilde{u} - u)\Big|}_{\leq \parallel D \parallel_\infty}  ~\txtd\mathcal{A}\mu^x_t(\tilde{u})\\
	&\leq C  \parallel D \parallel_\infty\mathcal{A}\mu^x_t (\mathbb{T}) \\
	&\leq C  \parallel D \parallel_\infty b\gamma_A.
	\end{align*}
	\textbf{(ii)} This follows from \textbf{(i)}, the fact that 
	\begin{equation*}
	T_{t,t_0}^xu = u + \int_{t_0}^t V[\mathcal{A},\mu,x](s,T_{s,t_0}^xu) ~\txtd s
	\end{equation*}
	and the compactness of $\Omega\times\mathcal{T}$. \\
	\textbf{(iii)} For any fixed $t_0 \in \mathcal{T}$, we define 
	\begin{equation*}
	\lambda(t):= \Big|T_{t,t_0}^xu - T^x_{t,t_0}w\Big|.
	\end{equation*}
	Using Lemma  \ref{lem: regularity of V} and the calculation in \textbf{(ii)} we get
	\begin{equation*}
	\lambda(t) \leq |u-w| + \mathcal{T} b \gamma_A \int_0^t \lambda(s) ~\txtd s.
	\end{equation*}
	Applying Gronwall's Lemma (cf. Lemma \ref{lem: Gronwalls lemma}) the claim follows.
	
\end{proof}

\begin{thm} \textbf{(Existence and Uniqueness of solutions for the fixed point equation)} 
	\label{thm: existence and uniqness of solutions for the fixed point eq} \\
	Assume $(\Omega,\Sigma,m)$ is a compact Borel probability space and that $A\in \mathcal{G}^m$ is a fixed graphop with corresponding measure $\nu^A$ on $\Omega^2$ in the sense of Theorem  \ref{thm: measure repres of graphops} and  family of fiber measures $\{\nu_x^A\}_{x \in \Omega}$. Then	following statements are true: \\
	\textbf{	(I)} For any initial condition $\bar{\mu}_0 \in \bar{\mathcal{M}}^{b}$,
	the map $\mathcal{F}:\mathcal{M}^{b}_{\mathcal{T}} \to \mathcal{M}^{b}_{\mathcal{T}}$ given by 
	\begin{equation}\label{eq: map F in thm exist of solut}
	\mathcal{F}\kappa^y_t := \mu_0^y \circ T_{0,t}[\mathcal{A},\kappa,y],  \quad \forall y\in \Omega,
	\end{equation}
	is well-defined and a contraction on $(\mathcal{M}^{b}_{\mathcal{T}},d_\alpha^{b})$ for any $\alpha > 2Cb + b\gamma_A$. \\
	\textbf{(II)} For any initial condition $\bar{\mu}_0 \in \bar{\mathcal{M}}^{b}$ there is a uique fixed point $\kappa \in \mathcal{M}^{b}_{\mathcal{T}}$ of the map $\mathcal{F}$, i.e., there is a unique $\kappa \in \mathcal{M}^{b}_{\mathcal{T}}$ satisfying
	\begin{equation*}
	\mathcal{F} \kappa = \kappa.
	\end{equation*}
	Furthermore, for any startpoint $\kappa^0 \in \mathcal{M}^{b}_{\mathcal{T}}$ the fixed point iteration given by
	\begin{equation*}
	\kappa^{n+1} := \mathcal{F}\kappa^n
	\end{equation*}
	converges to $\kappa$. 
\end{thm}

The proof of Theorem \ref{thm: existence and uniqness of solutions for the fixed point eq} uses exactly the same argument used in \cite[Theorem 2.4]{Medvedev2018} and is included in Appendix~\ref{sec: appendix, technical lemmas} for convenience. The previous lemmas and discussed extensions provide the main ingredients for the proof to succeed. In particular, the key steps were to design a suitable metric space setting to work with graphops, which we accomplished above. Likewise, a second result we immediately obtain is the following theorem:	

\begin{thm} \textbf{(Existence and uniqness of a weak solution for the  VFPE with graphops)}
	\label{thm: existence and uniqness for weak solutions of the VFPE with graphops}
	Assume $(\Omega,\Sigma,m)$ is a compact Borel probability space and that $A\in \mathcal{G}^m$ is a fixed graphop with corresponding measure $\nu^A$ on $\Omega^2$ in the sense of Theorem  \ref{thm: measure repres of graphops} and  family of fiber measures $\{\nu_x^A\}_{x \in \Omega}$. Moreover, assume that the initial condition $\bar{\mu_0} \in \bar{\mathcal{M}}^{b}$ is absolutely continuous with density $\rho\in L^\infty(\Omega,m)$, i.e. it holds
	\begin{equation*}
	\mu_0^y = \rho^0(u,y)~\txtd u  \text{\quad for all $x\in \Omega$}.
	\end{equation*}
	Then there is a unique weak solution $\rho$ of the IVP (\ref{eq: IVP for the VFPE}).
\end{thm}
Having provided a suitable new metric setting above, the proof of Theorem \ref{thm: existence and uniqness for weak solutions of the VFPE with graphops} can be deduced from~\cite[Theorem 3.2]{Medvedev2018}.

\subsection{Further regularity in $x$}
Additional smoothness assumptions for the graphop $A$ and the initial condition $\mu_0$ increase the regularity of the weak solution in $x$. In this section let us restrict ourselves to proving continuity in $x$. 

Let us assume that the graphop $A$ satisfies following continuity property 
\begin{equation} \label{eq: cont. pproperty for graphop A}
Af(x) \to Af(x_0) \quad \text{ for } x\to x_0, \quad \forall x_0\in \Omega, f \in C(\Omega). 
\end{equation}
%From now on, we call graphops satisfying (\ref{eq: cont. pproperty for graphop A}) \textit{c-continuous}.
We define the restricted space
\begin{align*}
\bar{\mathcal{M}}^{b,2}:= \{ \bar{\mu}: \Omega \to \mathcal{M}_f(\mathbb{T}):  &\bar{\mu} \text{ is measurable, } x\mapsto \mu^x(S) \text{ is continuous for all  $S \in \mathcal{B}(\mathbb{T})$ with $\lambda(\partial{S}) = 0$} \\ &\text{and} \sup_{x \in \Omega} \mu^x(\mathbb{T}) \leq b \} \quad \quad  \subset \bar{\mathcal{M}}^{b}.
\end{align*}
%and
%\begin{align*}
%\mathcal{G}^2 &:= \{ B \in \mathcal{PB}(\Omega,\Sigma,l): l \text{ is a probability measure}, B \text{ is a c-continuous graphop, } \gamma_B \leq 1\}, \\
%\mathcal{G}^{m,2} &:= \{ B \in \mathcal{PB}(\Omega,\Sigma,m): B \text{ is a c-continuous graphop }, \gamma_B \leq 1\}.
%\end{align*}
%On $\bar{\mathcal{M}}^{b,2}$ we define the metric \begin{equation}
%\bar{d}^{b,2} (\bar{\mu},\bar{\kappa} ):= \sup_{B \in \mathcal{G}} \bar{d}^{b,B,2} (\bar{\mu},\bar{\kappa} ) := \sup_{B \in \mathcal{G}^2} \sup_{x\in \Omega} \Big(\int_\Omega d_{BL}(\mu^y,\kappa^y) ~\txtd\nu_x^B(y) \Big)
%\end{equation}
It can be shown that  $(\bar{\mathcal{M}}^{b,2},\bar{d}^{b})$ is a complete metric space, cf.~Lemma  \ref{lem: Mb,c is complete}, so that
\begin{equation*}
\mathcal{M}_\mathcal{T}^{b,2} := C(\mathcal{T}, (\bar{\mathcal{M}}^{b,2}, \bar{d}^{b}))
\end{equation*} 
is also complete.

Further, we have following lemma
\begin{lem} \textbf{(Increased regularity in $x$ for the fixed-point operator)} \label{lem: increased regularity in $x$}\\
	\textbf{(i)} For the extended graphop $\mathcal{A}$ we have 
	\begin{equation*}
	\mathcal{A}(\bar{\mathcal{M}}^{b,2}) \subset \bar{\mathcal{M}}^{b,2}. 
	\end{equation*}
	%Further, for every $\bar{\mu}, \bar{\kappa} \in \bar{\mathcal{M}}^{b,2}$ holds
	%\begin{equation*}
	%\bar{d}^{b,2}(\mathcal{A}\bar{\mu}, \mathcal{A}\bar{\kappa}) \leq 	\bar{d}^{b,2}(\bar{\mu}, \bar{\kappa}).
	%\end{equation*}
	\textbf{(ii)} For any $\mu \in \mathcal{M}^{b,2}_\mathcal{T}$ and $x \in \Omega$, $t \in [0,T], u \in \mathbb{T}$, the characteristic field $V[\mathcal{A}, \mu,x](t,u)$ is continuous in $x$.\\
	\textbf{(iii)} The corresponding flow $T^x_{t,t_0}[\mu]u$ is for any $\mu \in \mathcal{M}^{b,2}_\mathcal{T}$ and $x \in \Omega$, $t \in [0,T], u \in \mathbb{T}$ continuous in $x$. \\
	\textbf{(iv)} For the fixed-point operator $\mathcal{F}$ defined in (\ref{eq: map F in thm exist of solut}) we have 
	\begin{equation*}
	\mathcal{F}(\mathcal{M}^{b,2}_{\mathcal{T}}) \subset \mathcal{M}^{b,2}_{\mathcal{T}}. 
	\end{equation*}
\end{lem}
\begin{proof}
	\textbf{(i)} Follows from (\ref{eq: cont. pproperty for graphop A}) and the definition of $\bar{\mathcal{M}}^{b,2}$. \\
	%The second one can be proven as in Lemma \ref{lem: Lips cont of A}.\\
	\textbf{(ii)} By \textbf{(i)} we have that the map 
	\begin{equation}
	x \mapsto \int_\mathbb{T} f(\tilde{u}) d \mathcal{A}\mu^x_t(\tilde{u})
	\end{equation}
	is continuous for any characteristic functions $f = \mathcal{\chi}_S$, $S \in \mathcal{B}(\mathbb{T})$. 
	By density, this property extends to the case $f(\tilde{u}) := D(\tilde{u}-u)$. \\
	\textbf{(iii)} Follows by \textbf{(ii)}.\\
	\textbf{(iv)} Follows by \textbf{(iii)} and the definition of $\bar{\mathcal{M}}^{b,2}$. 
\end{proof}
%Note also that in all statements from the previous Sections  \ref{sec: the extended graphop}, \ref{sec: the extended char field}, \ref{sec: the fixed point eq} we can replace the metric $\bar{d^b}$ by the metric  $\bar{d^{b,2}}$.
\begin{cor} \textbf{(Continuity of the solution of the fixed point equation in $x$)} \label{cor: conti of the solut of the fixed point eq in x}\\
	For any initial condition $\bar{\mu}_0 \in \bar{\mathcal{M}}^{b,2}$ there is a uique fixed point $\kappa \in \mathcal{M}^{b,2}_{\mathcal{T}}$ of the map $\mathcal{F}$, i.e., there is a unique $\kappa \in \mathcal{M}^{b,2}_{\mathcal{T}}$ satisfying
	\begin{equation*}
	\mathcal{F} \kappa = \kappa.
	\end{equation*}	
\end{cor}
\begin{proof}
	Follows directly from Lemma \ref{lem: increased regularity in $x$} \textbf{(iv)} and Theorem \ref{thm: existence and uniqness of solutions for the fixed point eq} \textbf{(II)} 
\end{proof}

\section{Continuous dependence of the solution of the fixed point equation on the graphop}
\label{sec: cont depen of the solution of the VFPE on the graphop}

After having established existence and uniqueness of solutions for the VFPE (\ref{eq: IVP for the VFPE}) we now want to ensure that small perturbations of the graphop will have small effect on the solution.
In this section we assume that  $(\Omega,\Sigma,m)$ is a compact Borel probability space with topology induced by a metric. 
 Recall that to the VFPE  (\ref{eq: IVP for the VFPE}) corresponds the fixed point equation (\ref{eq: fixed point eq assoc with the VFPE}). To compare solutions of this fixed point equation, we introduce following pseudometric on $\bar{\mathcal{M}}^b$:
\begin{equation}
\bar{d}^{b,m} = \int_\Omega d_{BL}(\mu^y,\kappa^y) ~\txtd m(y) =  \int_\Omega \sup_{f\in \mathcal{L}} \Big|f(v)d (\mu^y - \kappa^y)(v )\Big| ~\txtd m(y).
\end{equation}
We note that proofs could possible work in various different topologies, but this is beyond the scope of the current work, since we are only interested in the existence of a suitable topology, where continuous dependence holds.

\begin{lem} \textbf{(Estimation for varying measure)}
	\label{lem: estimations for varying measure for 1-1 bounded graphops} \\
	Let $\bar{\mu}, \bar{\kappa} \in \bar{\mathcal{M}}^b$ and $A\in \mathcal{G}^m$ with $\parallel A\parallel_{p\to q} <\infty$ for $p,q\in [1,\infty]$. Then, for any $t\in \mathcal{T}, u \in \mathbb{T}$ we have
	\begin{equation*}
	\int_\Omega	\Big| V[\mathcal{A}, \mu,x](t,u) - 	V[\mathcal{A}, \kappa,x](t,u)\Big|~\txtd m(x) \leq 	2C \parallel A \parallel_{p \to q} \Big(\int_\Omega d_{BL}(\mu^y_t,\kappa^y_t)^p~\txtd m(y) \Big)^{\frac{1}{p}}.
	\end{equation*}
	Especially, for $\parallel A \parallel_{1 \to q} < \infty$ we have
	\begin{equation*}
	\int_\Omega	\Big| V[\mathcal{A}, \mu,x](t,u) - 	V[\mathcal{A}, \kappa,x](t,u)\Big|~\txtd m(x) \leq 2C \parallel A \parallel_{1 \to q} \bar{d}^{b,m}(\bar{\mu}_t, \bar{\kappa}_t).
	\end{equation*}
\end{lem}

\begin{proof}
	By H\"{o}lder`s inequality we have $\parallel f \parallel_1 \leq \parallel f \parallel_q$ for any $f\in L^q(\Omega,m)$.
	Thus, 
	\begin{align*}
	\int_\Omega	\Big| V[\mathcal{A}, \mu,x](t,u) &- 	V[\mathcal{A}, \kappa,x](t,u)\Big|~\txtd m(x) \\&= C \int_\Omega \Big| A \Big( \int_{\mathbb{T}} D( \tilde{u} - u) (~\txtd\mu_t^{\cdot}(\tilde{u}) - ~\txtd\kappa_t^{\cdot}(\tilde{u})) \Big) (x)    \Big| ~\txtd m(x)  \\
	&\leq \Big(\int_\Omega \Big| A \Big( \int_{\mathbb{T}} D( \tilde{u} - u) (~\txtd\mu_t^{\cdot}(\tilde{u}) - ~\txtd\kappa_t^{\cdot}(\tilde{u})) \Big) (x)    \Big|^q ~\txtd m(x)\Big)^{\frac{1}{q}}\\
	&\leq C\parallel A \parallel_{p \to q} \Big(\int_\Omega \Big|\int_{\mathbb{T}} D( \tilde{u} - u) (~\txtd\mu_t^{x}(\tilde{u}) - ~\txtd\kappa_t^{x}(\tilde{u}))     \Big|^p ~\txtd m(x) \Big)^{\frac{1}{p}} \\
	&\leq  2C \parallel A \parallel_{p \to q} \Big(\int_\Omega d_{BL}(\mu_t^y,\kappa_t^y)^p~\txtd m(y) \Big)^{\frac{1}{p}}.
	\end{align*}
	This finishes the proof.		
\end{proof}

Now, for any $n\in \mathbb{N}$, let $A^n,A \in \mathcal{G}^m$ be graphops with corresponding canonical extentions $\mathcal{A}^n$, $\mathcal{A}$ on the space $\bar{M}^b$, as in (\ref{eq: extention of the measn field op on measures}). 
For an initial condition $\bar{\mu_0} \in \bar{\mathcal{M}}^{b}$, let $\mu^n, \mu \in \mathcal{M}^b_\mathcal{T}$ be the solutions of following fixed point equations
\begin{subequations}\label{eq: fixed point equations 2}
	\begin{alignat}{4}
	\mu_t^{n,y} &= \mu_0^y \circ T_{0,t}[\mathcal{A}^n,\mu^n,y],  \text{ \quad for all $y \in \Omega$,} \\
	\mu_t^y &= \mu_0^y \circ T_{0,t}[\mathcal{A},\mu,y],  \text{ \quad for all $y \in \Omega$,}
	\end{alignat}
\end{subequations}
cf.~Theorem \ref{thm: existence and uniqness of solutions for the fixed point eq}. 

\begin{prop} \textbf{(Continuous dependence of fixed point solutions on graphops)} 
	\label{prop: continuous dependence on the graphop II}\\
	Assume that $A^n \to_o A$ and $\parallel A \parallel_{1\to q} <\infty$, for a $q\in [1,\infty]$. Further assume that for any Borel set $S\subset \mathbb{T}$ with $\lambda(\partial S) = 0$ and $t\in [0,T]$ the function $x\mapsto\mu^x_t(S)$ is bounded and has discontinuity set $U$ with $\mu(U)=0$. Then,
	\begin{equation}
	\sup_{t \in \mathcal{T}} \bar{d}^{b,m}(\bar{\mu}^n_t,\bar{\mu}_t ) \to 0 \quad \text{ as } n \to \infty.
	\end{equation}	
\end{prop}

\begin{proof}
	The proof starts as in \cite[Lemma 2.7]{Medvedev2018} and then uses a new argument. Following the same steps as in the proof of Theorem \ref{thm: existence and uniqness of solutions for the fixed point eq} (equation (\ref{eq: thm: A is a constaction, 1})) we can show that 
	\begin{align*}
	\bar{d}^{b,m}(\bar{\mu}_t,\bar{\mu}^n_t) &\leq \int_\Omega \int_{\mathbb{T}} | T^y_{t,0}[\mu]v - T^y_{s,0}[\mu^n]v |~\txtd\mu_0^y(v) ~\txtd m(y)=: \lambda(t) \\
	&\leq  \int_0^t\int_\Omega \int_{\mathbb{T}} | V[\mathcal{A},\bar{\mu},y] (T_{s,0}^y[\mathcal{A},\bar{\mu}]v,s) -V[\mathcal{A}^n,\bar{\mu}^n,y] (T_{s,0}^y[\mathcal{A}^n,\bar{\mu}^n]v,s) |~\txtd\mu_0^y(v) ~\txtd m(y) ~\txtd s.  
	\end{align*}	
	Using the triangle inequality we have
	\begin{equation*}
	\lambda(t) \leq \lambda_1(t) + \lambda_2(t) + \lambda_3(t),
	\end{equation*}
	with 
	\begin{align*}
	\lambda_1(t) &= \int_0^t\int_\Omega \int_{\mathbb{T}} | V[\mathcal{A},\bar{\mu},y] (T_{s,0}^y[\mathcal{A},\bar{\mu}]v,s) -V[\mathcal{A},\bar{\mu}^n,y] (T_{s,0}^y[\mathcal{A},\bar{\mu}]v,s) |~\txtd\mu_0^y(v) ~\txtd m(y) ~\txtd s, \\
	\lambda_2(t) &= \int_0^t\int_\Omega \int_{\mathbb{T}} | V[\mathcal{A},\bar{\mu}^n,y] (T_{s,0}^y[\mathcal{A},\bar{\mu}]v,s) -V[\mathcal{A}^n,\bar{\mu}^n,y] (T_{s,0}^y[\mathcal{A},\bar{\mu}]v,s) |~\txtd\mu_0^y(v) ~\txtd m(y) ~\txtd s, \\
	\lambda_3(t) &= \int_0^t \int_\Omega \int_{\mathbb{T}} | V[\mathcal{A}^n,\bar{\mu}^n,y] (T_{s,0}^y[\mathcal{A},\bar{\mu}]v,s) -V[\mathcal{A}^n,\bar{\mu}^n,y] (T_{s,0}^y[\mathcal{A}^n,\bar{\mu}^n]v,s) |~\txtd\mu_0^y(v) ~\txtd m(y) ~\txtd s.
	\end{align*}	
	For the first term we obtain, using Lemma \ref{lem: estimations for varying measure for 1-1 bounded graphops},
	\begin{align*}
	\lambda_1(t) &\leq 2Cb \parallel A \parallel_{1 \to q} \int_0^t\int_\Omega d_{BL}(\mu_s^y,\mu_{s}^{n,y}) ~\txtd m(y) ~\txtd s =2Cb \parallel A \parallel_{1 \to q} \int_0^t\bar{d}^{b,m}(\bar{\mu}_s,\bar{\mu}^n_s) ~\txtd s.
	\end{align*}
	For the third we get by Lemma \ref{lem: regularity of V} that
	\begin{equation*}
	\lambda_3(t) \leq b \int_0^t\lambda(s)~\txtd s.
	\end{equation*}
	All in all, we find	
	\begin{equation*}
	\lambda(t) \leq 2Cb \parallel A \parallel_{1 \to q}\int_0^t\bar{d}^{b,m}(\bar{\mu}_s,\bar{\mu}^n_s) ~\txtd s +  \lambda_2(t) + b \int_0^t\lambda(s)~\txtd s.
	\end{equation*}
	Hence, by Gronwall's inequality (cf. Lemma \ref{lem: Gronwalls lemma}) we have
	\begin{equation*}
	\bar{d}^{b,m}(\bar{\mu}_t,\bar{\mu}^n_t)\leq \txte^{bt} \Big( C_1 \int_0^t  \bar{d}^{b,m}(\bar{\mu}_s,\bar{\mu}^n_s)  \txte^{-bs} ~\txtd s + \lambda_2(t) \Big),
	\end{equation*}
	with $C_1:= 2Cb \parallel A \parallel_{1 \to q}$. Defining $\phi(t):= \txte^{-bt}\bar{d}^{b,m}(\bar{\mu}_t,\bar{\mu}^n_t)$ and applying Gronwall's inequality for a second time we see that 
	\begin{equation*}
	\phi(t) \leq \txte^{C_1t} \lambda_2(t),
	\end{equation*}
	which implies that 
	\begin{equation}\label{eq: prop cont de on the graphop 1}
	\sup_{t \in \mathcal{T}} \bar{d}^{b,m}(\bar{\mu}^n_t,\bar{\mu}_t ) \leq \lambda_2(T)\txte^{(C_1 + b)T}.
	\end{equation}
	Thus, we have to deal with the term $\lambda_2(T)$. From $A^n \to_o A$ and
	Remark \ref{rem: pointwise convergence and graphop convergence}
	% the definition of weak convergence \cite[Definition 13.12]{KlenkeBook2014} 
	follows immediately that for any Borel set $S\subset \mathbb{T}$ with $\lambda(\partial S) = 0$ we have
	\begin{equation*}
	\mathcal{A}^n\mu^x_t(S) = \int_\Omega\mu^y_t(S) d\nu_x^n(y)  \to \mathcal{A}\mu^x_t(S) \quad  \text{ for all $t \in [0,T]$ and $m$-a.e. } x\in \Omega.
	\end{equation*}
	(Note that for this step we needed the a.e. continuity of the map $\Omega \to \mathbb{R}$, $x\mapsto\mu^x_t(S)$  for any Borel set $S\subset \mathbb{T}$ with $\lambda(S)=0$.) 
	By Portmanteau's theorem (see for instance \cite[Theorem 13.16]{KlenkeBook2014}) 
%Again by Remark \ref{rem: pointwise convergence and graphop convergence}
 this implies that 
	\begin{equation*}
	\mathcal{A}^n\mu^x_t \to_w \mathcal{A}\mu^x_t \quad  \text{ for all $t \in [0,T]$ and $m$-a.e. }  x\in \Omega.
	\end{equation*}
	Hence, due to the equivalence of the L\'evy-Prokhorov metric with the bounded Lipschitz distance~\cite{Gibs07} we have
	\begin{equation*}
	d_{BL}(\mathcal{A}^n\mu^x_t , \mathcal{A}\mu^x_t ) \to 0 \quad \text{ as $n\to \infty$ for all $t \in [0,T]$ and  $m$-a.e. } x\in \Omega,
	\end{equation*}
	which in turn implies by the dominated convergence theorem (applicable due to the compactness of $\Omega\times[0,T]$) that 
	\begin{equation*}
	\int_0^T\int_\Omega\int_\mathbb{T} d_{BL}(\mathcal{A}^n\mu^x_s , \mathcal{A}\mu^x_s) ~\txtd\mu_0^y(v)~\txtd m(y) ~\txtd s \to 0 \quad \text{ as } n\to \infty.
	\end{equation*}
	Thus, since
	\begin{align*}
	\lambda_2(t) &= \int_0^t\int_\Omega \int_{\mathbb{T}} | V[\mathcal{A}^n,\bar{\mu},y] (T_{s,0}^y[\mathcal{A}^n,\bar{\mu}^n]v,s) -V[\mathcal{A},\bar{\mu},y] (T_{s,0}^y[\mathcal{A}^n,\bar{\mu}^n]v,s) |~\txtd\mu_0^y(v) ~\txtd m(y) ~\txtd s\\
	&\leq  \int_0^T\int_\Omega \int_{\mathbb{T}} \underbrace{\Big|\int_\mathbb{T} D(\tilde{u} - T^y_{s,0}[\mathcal{A}^n,\bar{\mu}^n]v) d(\mathcal{A}^n\mu^y_s-\mathcal{A}\mu^y_s)(\tilde{u})\Big|}_{\leq 2 d_{BL}(\mathcal{A}^n\mu^y_s , \mathcal{A}\mu^y_s)}~\txtd\mu_0^y(v) ~\txtd m(y) ~\txtd s
	\end{align*}.
	we see that
	\begin{equation}\label{eq: prop cont de on the graphop 2}
	\lim_{n\to \infty} \lambda_2(T) = 0.
	\end{equation}	
	Combining (\ref{eq: prop cont de on the graphop 1}) with (\ref{eq: prop cont de on the graphop 2}) finishes the proof.
\end{proof}

Having completed all the necessary existence, uniqueness, and continuous dependence results, we now know that VFPEs involving graphops are suitably well-posed. The next step is to show that they are indeed mean-field limits for the generalized Kuramoto model on graphs.

\section{Graphon approximation of graphops on compact groups}
\label{sec: graphon approx of graphops on compact groups}

Our firs step is to find, for a given graphop $A$, a suitable graphon regularization $A^n$, which approximates $A$ in the sense of Definition \ref{defn: o-graphop conv}, where we defined our main new topology adapted to graphops. To find the right approximation via graphops, a key idea is to employ Fourier methods. We recall the notions we need briefly. 

A \textit{locally compact abelian (LCA) group} is an abelian group $G$ which is a locally compact Hausdorff space and such that the group operations are continuous (in other words, $G$ is topological group which is abelian, locally compact and Hausdorff). To be more precise, the maps
\begin{subequations}
	\begin{alignat*}{4}
	G\to G, &\quad x\mapsto -x, \\
	G\times G \to G &\quad (x,y) \to x+y,
	\end{alignat*}
\end{subequations}
are both continuous. Standard examples for LCA groups are $\mathbb{R}^d$ and $\mathbb{T}^d$ with the usual topologies and $(\mathbb{Z},+)$ with the discrete topology.

\begin{defn}
	A \textit{Haar measure} $\mu_G$ on a locally compact group $G$ is a positive, regular, Borel measure having the following two properties:\\
	\textbf{(i)} $\mu_G$ is finite on compact sets, i.e., we have
	\begin{equation*} 
	\mu_G(E) < \infty \quad \text{ if $E$ is compact};
	\end{equation*}
	\textbf{(ii)} $\mu_G$ is invariant under translation, i.e., we have
	\begin{equation*}
	\mu_G(x+E) = \mu_G(E) \text{\quad for all measurable $E \subset G$ and all $x\in G$.}
	\end{equation*}
\end{defn}

One can prove that the Haar measure always exists and is unique up to multiplication by a positive constant. One can also prove that the Haar measure is finite if and only if $G$ is compact, see \cite{Katznelson2004}. In the following, in the case that $G$ is a \textit{Compact Abelian (CA) group} we always assume that $\mu_G$ is the normalized probability measure. 

\begin{defn}\textbf{(summability kernel on LCA group)}\\
	A \textit{summability kernel} on the LCA group $G$ is a sequence $\{k_n\}_{n \in \mathbb{N}}$ satisfying the following conditions:\\
	\textbf{(1)} \begin{equation*}
	\int_{G} k_n(x) ~\txtd\mu_G(x) =1.
	\end{equation*}
	\textbf{(2)} \begin{equation*}
	\int_{\mathbb{T}} |k_n(x)|~\txtd\mu_G(x) \leq const.
	\end{equation*}
	\textbf{(3)} For any neighborhood $V$ of $0$ in $G$ we have
	\begin{equation*}
	\lim_{n\to \infty} \int_{G\setminus V}|k_n(x)| ~\txtd\mu_G(x) =0.
	\end{equation*}	
\end{defn}

Furthermore, we say that the summability kernel $\{k_n\}_{n \in \mathbb{N}}$ is:\\ 
\textbf{(i)} \textit{positive}, provided that $k_n(x) \geq 0$ for all $x$ and $n$. \\
\textbf{(ii)} \textit{symmetric}, provided that $k_n(x) = k_n(-x)$ for all $x \in G$. \\

Standard examples for symmetric and positive summability kernels are the Poisson and the Gauss kernels. We note that we can view any given summability kernel $k_n$ as a P-operator $K_n:L^\infty(G,\mu_G)\to L^1(G,\mu_G)$, with action given by the convolution
\begin{equation}
K_nf(x):= k_n*f(x) =  \int_Gk_n(x-y)f(y)~\txtd\mu_G(y). 
\end{equation} 
If moreover $\{k_n\}_{n \in \mathbb{N}}$ is a positive and symmetric summability kernel, then the function $W^n:G\times G \to \mathbb{R}$ given by $W^n(x,y):= k_n(x-y)$ is a graphon and the associated graphop $A_{W^n}: L^\infty(G,\mu_G) \to L^1(G,\mu_G)$ is given by the convolution
\begin{equation}
\label{eq: convolution with summability kernel}
A_{W^n}f(x) =\int_{\mathbb{T}} k_n(x-y)f(y)~\txtd\mu_G(y) = K_nf(x).
\end{equation}

\begin{thm}\textbf{(Approximation by summability kernels)}
	\label{thm: summability kernel = approximate identity}\\
	Let $\{k_n\}_{n \in \mathbb{N}}$ be a summability kernel on the CA group $G$. Then, for every $f\in L^p(G,\mu_G)$, $1 \leq p <\infty$, we have 
	\begin{equation}
	\parallel K_nf - f\parallel_{L^p} =0.
	\end{equation}
	Moreover, if $f\in C(G)$ then the convergence is uniform, i.e., 
	\begin{equation}
	\parallel K_nf - f\parallel_{\infty} =0.
	\end{equation}
\end{thm}

\begin{proof}
	This is a classical result in harmonic analysis. See for example \cite[Chapter 7.2, Theorem 2.11]{Katznelson2004} and references therein.
\end{proof}

\begin{lem}\textbf{(graphops preserve uniform convergence of continuous functions)} 
	\label{lem: graphops preserve uniform convergence}\\
	Let $A: L^\infty(G,\mu_G)\to L^1(G,\mu_G)$ be a graphop with $\gamma_A<\infty$. 	
	Then, for any sequence $\{f_n\}_{n\in \mathbb{N}}$  and any $f$,  with $f_n\in C(G)$ we have
	\begin{equation*}
	\lim_{n\to \infty } \parallel f_n-f \parallel_{\infty} =0 \quad\Rightarrow \quad \lim_{n\to \infty } \parallel Af_n- Af \parallel_{\infty} =0.
	\end{equation*}
\end{lem}

\begin{proof}
	Due to uniform convergence, $f$ is continuous.
	Thus, since $f_n ,f \in C(G)$ we have that 
	\begin{equation*}
	|f_n(x) - f(x)| \leq \parallel f_n - f \parallel_\infty \quad \text{ for all } x \in G. 
	\end{equation*}
	Thus,
	\begin{align*}
	\parallel	A f^n - A f\parallel_\infty &= \sup_{x \in \Omega}\Big|\int_\Omega (f^n(y) - f(y))
	~\txtd\nu_x(y) \Big|\\
	&\leq \parallel f^n - f\parallel_\infty \underbrace{\sup_{x\in G} \nu_x(G)}_{=\gamma_A} \\
	&\to 0 \quad \text{ as } n\to \infty,
	\end{align*}
	and the result follows.	
\end{proof}

\begin{prop}\textbf{(o-graphon approximability for graphops on  $L^\infty(\Omega,m)$)} 
	\label{prop: o graphon approx using summabi kernels}\\
	Assume that $A: L^\infty(G,\mu_G)\to L^1(G,\mu_G)$ is a graphop on the CA group $G$ with Haar measure $\mu_G$. Assume that $\gamma_A< \infty$ and $\{k_n\}_{n \in \mathbb{N}}$ is a positive and symmetric summability kernel. Then, the regularization $K_nAK_n$ defines a sequence of graphons such that 
	\begin{equation*}
	K_n A  K_n \to_o A.
	\end{equation*}
\end{prop}
\begin{proof}
	Let $f:G\to \mathbb{R}$ be any continuous function. 
	Using Fubini's theorem we compute that
	\begin{align*}
	K_n A  K_n f(x) &= \int_{G}k_n(x-y) \int_{G}\int_{G}k_n(\hat{z}-z)f(z)d\mu_G(z)~\txtd\nu_y(\hat{z})d\mu_G(y) \\
	&= \int_{G} \underbrace{\int_{G} k_n(x-y)\int_{G}k_n(\hat{z}-z)~\txtd\nu_y(\hat{z})d\mu_G(y)}_{=: W^n(x,z)} f(z) d\mu_G(z). 
	\end{align*}
	Since the kernel $W^n$ can be rewriten as
	\begin{equation}\label{eq: graphon approximation using summability kernel}
	W^n(x,z) = \int_{G}\int_{G} k_n(x-y)k_n(z-\hat{z})~\txtd\nu(y,\hat{z}),
	\end{equation}
	which is a symmetric, bounded and positive function, we have that $K_n A  K_n$ is for all $n \in \mathbb{N}$ a graphon. By Theorem \ref{thm: summability kernel = approximate identity} we have that $K_nf\to f$ uniformly which implies by Lemma \ref{lem: graphops preserve uniform convergence} that $AK_nf \to Af$ uniformly, since $K_nf$ is continuous.  Thus, we have for any $x\in G$
	\begin{equation*}
	|K_nA K_n f(x) - Af(x)| \leq \parallel A^nf - Af\parallel_{\infty}\int_\mathbb{T}|k_n(x-y)|d\mu_G(y) = \parallel A^nf - Af\parallel_{\infty} \to 0.
	\end{equation*}
%	
%	Now, for any bounded function  $g:G\to \mathbb{R}$ with discontinuity set $U_g$ with $\mu_G(U_g)=0$ and arbitary $\epsilon>0$ we can find a continuous function  $f:G\to \mathbb{R}$ ($f = f(\epsilon, g)$) such that $\parallel f-g\parallel_\infty<\epsilon$. For this $f$ we have by the previous step for all $x\in G$
%	\begin{equation}	\label{eq: o graphon approx using summabi kernels 1}
%	|K_nA K_n g(x) - g(x)| \leq |K_nA K_n (g - f)(x)| + \parallel K_nA K_n f - f\parallel_\infty + \parallel f- g \parallel_\infty.
%	\end{equation}
%	The second term in the last inequality can be controlled by the previous step. For the first term we have
%	\begin{align}	\label{eq: prop o graphon approx using summabi kernels 2}
%	|K_nA K_n (g - f)(x)| &= \Big|\int_G \int_G\int_G k_n(x-y)k_n(z-\hat{z}) d\nu(y,\hat{z}) (g-f)(z)d\mu_G(z) \Big| \nonumber\\
%	&\leq \Big |\int_G \int_G\int_G k_n(x-y)k_n(z-\hat{z}) d\nu_y(\hat{z})d\mu_G(y) d\mu_G(z) \Big| \parallel g-f\parallel_\infty \nonumber\\
%		&\leq  \epsilon\Big|\int_G  k_n(x-y)\int_G\underbrace{\int_Gk_n(z-\hat{z})d\mu_G(z)}_{=1} d\nu_y(\hat{z})d\mu_G(y) \Big| \nonumber\\
%		&\leq \epsilon\gamma_A. 
%	\end{align}
%	
%	All in all, by (\ref{eq: o graphon approx using summabi kernels 1}) and (\ref{eq: prop o graphon approx using summabi kernels 2}), for every $\epsilon>0$ there exists an $N\in \mathbb{N}$ such that for all $n\geq N$ we have
%	\begin{equation}
%	\parallel K_ng - g\parallel_\infty \leq \epsilon\gamma_A + 2\epsilon. 
%	\end{equation}	
	The claim follows now by Remark \ref{rem: pointwise convergence and graphop convergence}.
\end{proof}

\section{Mean field Approximation}
\label{sec: mean field approximation}

We now have everything we need to solve the main problem, which is to show that the VFPE (\ref{eq: IVP for the VFPE^infty}) with the graphop $A$ approximates the discrete Kuramoto problem (\ref{eq: Kuramoto problem graphon approximable graphops}). 

We come back to the setup of Section~\ref{sec: main problem}. For the compact Borel probability space $(\Omega,\Sigma,m)$ we assume additionaly that the topology is induced by a metric and that $\Omega=G$ is a CA group and $ m= \mu_G$ is the Haar probability measure. $A \in \mathcal{G}^{m}$ is assumed to be a graphop 
$A: L^\infty(G, \mu_G) \to L^1(G, \mu_G)$ with $\parallel A \parallel_{1\to q} < \infty$ for a $q\in [1,\infty]$. Recall that by Proposition \ref{prop: o graphon approx using summabi kernels} there exists a sequence of graphons $A^K: L^\infty(G, \mu_G) \to L^1(G, \mu_G)$ with graphon kernels $W^K$, such that $A^K\to_o A$. We assume that the weigths $A^{N,K}$ in (\ref{eq: weights Kurmato on T}) are given by these kernels.
We assume furthermore that the graphop $A$ satisfies the continuity property  (\ref{eq: cont. pproperty for graphop A}).
Recall that for any $n, M\in \mathbb{N}$ we use the notation $N:=nM$ and we have that $\Omega^N_{(i-1)M + k} \subset \Omega^n_i$ for all  $i \in [n], k \in [M]$ and $m(\Omega^n_i) = \frac{1}{n}$.
Further we assume for  the initial values $u^{N,K}_{(j-1)M +l}(0) = u^{0}_{n, jl} , j\in [n], l\in [M]$ that for every $j\in \mathbb{N}$,  $u^{0}_{n, jl}, l\in \mathbb{N}$ are independent identically distributed continuous random variables with density $\rho^0_{n,j}(v)= n \int_{\Omega^n_j} \rho^0(v,x)dm(x)$. Here, for the initial condition $\rho^0$ we assume that 
\begin{equation}
\rho^0(u,x) \in L^1(\mathbb{T} \times \Omega), \quad \rho^0 \geq 0, \quad \int_\mathbb{T} \rho^0(u,x) du = 1 \text{ $m$-a.e. $x\in \Omega$}, \quad  \rho^0(u,\cdot) \in C(\Omega).
\end{equation}
Finally (as in Section 3.3 of \cite{Medvedev2018}) we define the probability space $ \Big(X^{(n)}:= (\mathbb{R}^\infty)^n, \mathcal{F}^{(n)} = \mathcal{B}(\mathbb{R}^\infty)^n, \mathbb{P}_0^{(n)} \Big)$ with the product measure $\mathbb{P}_0^n = \prod_{j \in \mathbb{N}}\rho^0_{n,j}dm$.
%Further, let $\xi^N_{(i-1)M+k}$ be a sequence of points such that for an $i \in [n]$, $\xi^N_{(i-1)M+k} \in \Omega^n_i$ are for all $k \in[M]$ independent and identical distributed according to $m$ restricted to $\Omega^n_i$, and assume that the initial values $u^{N,0}_j \in \mathbb{T}, j\in [N]$ are independent random variables, whose distribution have densities $\rho^0(\cdot,\xi^N_{j})$, $j\in [N]$ (w.r.t. to the Lebesque measure on $\mathbb{T}$). Assume additionally that for the initial condition $\rho^0$ is continuous in $x$. 
\begin{thm}\textbf{(VFPE approximates the discrete Kuramoto's model)}
	\label{thm: vfpe approximates the discrete Kuramoto's problem for graphon approximable graphops}\\
	Under the previous assumptions, for any given $\epsilon>0$ there exists a $K_1 \in \mathbb{N}$ such that for any $K\geq K_1$ there exist $N_1(K,\epsilon)$ such that for all 
	$n \geq N_1(K)$ there exist $\mathcal{U}^{(n)} \in \mathcal{F}^n, \mathbb{P}_0^{(n)}(\mathcal{U}^{(n)})=1$ and $M_1= M_1(\epsilon, n_1, K)$ we have
%$M_1(K), N_1(K) \in \mathbb{N}$ such that for all  $M\geq M_1(K), n \geq N_1(K)$ we have
	\begin{equation}
	\label{eq: statemnt in thm vfpe approximates the discrete Kuramoto's problem}
	\sup_{t\in \mathcal{T}} \bar{d}^{b,m}(\bar{\nu}_{n,M,K,t}, \bar{\nu}_t)< \epsilon
	\end{equation}
	for every $M \geq M_1$ and $u^0 \in \mathcal{U}^{(n)}$.	
\end{thm}

\begin{proof} Let $\epsilon >0$. Since, by Proposition \ref{prop: o graphon approx using summabi kernels}, $A^K\to_o A$, by Proposition \ref{prop: continuous dependence on the graphop II} (cf. also Corollary \ref{cor: conti of the solut of the fixed point eq in x}) we can find a $K_1\in \mathbb{N}$ such that for all $K\geq K_1$ we have
	\begin{equation}
	\sup_{t\in \mathcal{T}} \bar{d}^{b,m}(\bar{\nu}^K_{t}, \bar{\nu}_t)< \frac{\epsilon}{2}.
	\end{equation}
	Furthermore, by \cite[Theorem 3.11]{Medvedev2018} (taking into account also Remark \ref{rem: graphons on Omega}) we can find for any $K \in \mathbb{N}$ an $N_1(K,\epsilon)$ such that for all 
	$n \geq N_1(K)$ there exist $\mathcal{U}^{(n)} \in \mathcal{F}^{(n)}, \mathbb{P}^{(n)}(\mathcal{U}^n)=1$ and $M_1= M_1(\epsilon, n_1, K)$ such that
	% $M_1(K), N_1(K),\in \mathbb{N}$ such that for all $M\geq M_1(K), n \geq N_1(K)$
	\begin{equation}
	\sup_{t\in \mathcal{T}} \bar{d}^{b,m}(\bar{\nu}_{n,M,K,t}, \bar{\nu}_t^K)< \frac{\epsilon}{2}
	\end{equation}
	for every $M \geq M_1$ and $u^0 \in \mathcal{U}^n$. 
	Thus, by the triangle inequality we have for all  $M\geq M_1(K), n \geq N_1(K)$: 
	\begin{equation}
	\sup_{t\in \mathcal{T}} \bar{d}^{b,m}(\bar{\nu}_{n,M,K,t}, \bar{\nu}_t) \leq \sup_{t\in \mathcal{T}} \bar{d}^{b,m}(\bar{\nu}_{n,M,K,t}, \bar{\nu}_t^K) + \sup_{t\in \mathcal{T}} \bar{d}^{b,m}(\bar{\nu}^K_{t}, \bar{\nu}_t) <
	\epsilon \quad { a.s. },
	\end{equation}	
	which finishes the proof.
\end{proof}

\begin{rem}\label{rem: graphons on Omega}
 \cite[Theorem 3.11]{Medvedev2018} is stated for graphons defined on $I$, with the special partition $\{0, \frac{1}{n}, \frac{2}{n},...,1\}$, but the argument used in the proof holds for arbitary graphons on the space $\Omega$ (in the sense of Definition \ref{def: graphons}) and for arbitary paritions as stated in our assumptions.
 % More precisely, the general restatement of \cite[Lemma 3.3, Lemma 3.4 and Theorem 3.9]{Medvedev2018} is obvious; The only (minimal) change is done in the statement of \cite[Theorem 3.6]{Medvedev2018} and in the proof of \cite[Lemma 3.8]{Medvedev2018}: There they assume that the initial conditions are distributed according to $\rho^0(\cdot, jN^{-1})$ which is the only place where they take into account the special partition $\{0, \frac{1}{n}, \frac{2}{n},...,1\}$. This is the reason we needed to assume that in our case, the initial conditions $u^{N,0}_j \in \mathbb{T}, j\in [N]$ are independent random variables, whose distributions have densities $\rho^0(\cdot,\xi^N_{j})$, $j\in [N]$: this assumption allow us to conclude, as in \cite[Lemma 3.8, Equation (3.27)]{Medvedev2018}, that 
% for any continuous function $f$ on $\mathbb{T}$ we have by the Law of Large Numbers 
 %\begin{equation*}
 %\lim_{m\to \infty} M^{-1}\sum_{j=1}^M f(u^{N,0}_{(i-1)M + j}) = n \int_{\Omega^n_i} \int_\mathbb{T} f(v) \rho^0(v, \xi) dv dm(\xi), \quad  a.s.
% \end{equation*}
% and there is no further difficulty to continue the proof in  \cite[Lemma 3.8]{Medvedev2018}.  
\end{rem}

A big class of graphops, which satisfy all conditions of the previous theorem are the $c$-regular graphops:

\begin{cor}\textbf{(Mean field approximation for c-regular graphops on $L^\infty(G, \mu_G)$)} 
	\label{cor: Mean field approximation for c-regular graphops}\\
	If $A: L^\infty(G, \mu_G) \to L^1(G, \mu_G)$ is a $c$-regular graphop with $c\leq 1$, which satisfies the continuity property  (\ref{eq: cont. pproperty for graphop A}), then the statement of Theorem \ref{thm: vfpe approximates the discrete Kuramoto's problem for graphon approximable graphops} is satisfied. Especially, the claim holds for any Markov graphop on $L^\infty(G, \mu_G)$ satisfying the continuity property (\ref{eq: cont. pproperty for graphop A}).
\end{cor}

\begin{proof}
	Since the graphop $A$ is $c$-regular, we check that for all $x\in \Omega$ we have
	\begin{equation*}
	\nu_x(G) = A\chi_G(x) = c,
	\end{equation*}
	which implies that
	\begin{equation*}
	\gamma_A = \sup_{x \in G} \nu_x^A(G) = c \leq 1.
	\end{equation*}
	Thus, $ A\in \mathcal{G}^m$. Further, see $A$ is self-adjoint, we have for any $f\in L^1(G, \mu_G)$:
	\begin{equation}
	\int_\Omega \Big|Af(x) \Big|~\txtd x = \int_G\Big| f(x) \underbrace{A\chi_G(x)}_{= c\chi_G} \Big|dx  = c \int_G \Big|f(x) \Big| ~\txtd x,
	\end{equation}	
	which implies that 
	\begin{equation*}
	\parallel A \parallel_{1\to 1} = c < \infty.
	\end{equation*}
	Thus, all assumptions in Theorem \ref{thm: vfpe approximates the discrete Kuramoto's problem for graphon approximable graphops} are satisfied. 
\end{proof}

A concrete example for a graphop satisfying the assumptions is the so called spherical graphop: 

\begin{ex} \textbf{(Kuramoto's model on the spherical graphop)} \\
	We consider the Borel probability space $\Omega:= \mathbb{S}^2:= \{(x,y,z) \in \mathbb{R}^3: x^2+y^2+z^2 =1\}$ with the	uniform measure $\mu$. The \textit{spherical graphop}, discussed in \cite{Scededgy2018}, is the graphop $S:L^\infty(\Omega,\mu) \to L^1(\Omega,\mu)$ given by
	\begin{equation}
	(Sf)(a) := \int_{S^a} f(b) ~\txtd\nu_a(b),
	\end{equation}
	where $S^a$ denotes the set of normalized vectors in $\mathbb{R}^3$, which are orthogonal to $a$ (which is a circle on the sphere) and $\nu_a$ is the uniform measure on $S^a$; note that we can also view $a$ as a normal vector corresponding to the circle $S^a$. In other words, the spherical graphop defines the graph for which any point on the sphere is connected only to its orthogonal vectors. Hence, every neighborhood is one-dimensional. The spherical graphop is Markov graphop, which is neither a graphon (dense graph) nor a graphing (sparse graph); it is a prototypical example for an infinite graph of intermediate density. It is not difficult to see that the sphere $\mathbb{S}^2$ with the usual topology and vector addition  is a CA group, and the Haar measure $\mu_{\mathbb{S}^2}$ is simply the uniform measure on $\mathbb{S}^2$. Hence, the spherical graphop  satisfies all assumptions of Corollary \ref{cor: Mean field approximation for c-regular graphops} (and Theorem \ref{thm: vfpe approximates the discrete Kuramoto's problem for graphon approximable graphops}).	
\end{ex}

%%%%%%%%%%%%%%%%%%%%%%%%%%%%%%%%%%%%%%%%%%%%%%%%%%%%%%%%%%%%%%%%%%%%%%%%%%%%%%%%%%%%%%%%%%%%%%
\section{Conclusion and Outlook}\label{sec: conclusion and outlook}
\label{sec:conclusion}

Previous results on mean-field approximation for the discrete Kuramoto model on finite networks (\ref{eq: kuramotos model on graphs intro}) were restricted to the case that the limiting graph is a graphon, i.e., a dense structure. In this paper we have significantly generalized and extended them to the case that the finite graphs converge towards graphs of intermediate density, or even sparse graphs. Our approach was based upon the operator-functional representation of graphs via graphops provided recently in \cite{Scededgy2018}. Introducing tools from harmonic analysis we were able to approximate any graphop, defined on a CA group with the Haar measure, by graphons. Since for any graphon mean field approximation is guaranteed, we could then bypass working directly with the generalized Kuramoto model. With this idea we managed to prove mean field approximation for a big class of graphops defined on a CA group with the Haar measure, which contains any $c$-regular graphop and any Markov graphop. Furthermore, we showed existence, uniqueness, and continuous dependence on the graphop for the limiting PDE.

As already mentioned in the beginning, to simplify the notation and calculations, we have always assumed that in the discrete Kuramoto model  (\ref{eq: kuramotos model intro}) all initial frequencies $\omega_i$ are the same: $\omega_i = 0$, $\forall i\in [N]$, but all of our results extend to the general case of a frequency distribution. In this case, the Kuramoto model (\ref{eq: kuramotos model on graphs intro}) reads as
\begin{equation}
\dot{u}_i = \omega_i + \frac{C}{N}\sum_{j=1}^NA^N_{i,j}D(u_j - u_i), \quad i \in [N]:= \{1,2,...,N\}.
\end{equation} 
The simplification $\omega_i = 0$, $\forall i\in [N]$ had as a result that the equation of characteristics was defined on the one-dimensional compact space $\hat{G}= \mathbb{T}$. In the general case of of distributed frequencies $\{\omega_i\}_{i=1,...,N}$, we set $\hat{G} = \mathbb{T}\times \mathbb{R}$ and the characteristic field is given by 
\begin{equation*}
V[\mathcal{A},\mu,x](t,P):= \begin{pmatrix}\omega + C\int_{\hat{G}}g(\tilde{\omega})D(\tilde{u} - u)~\txtd(\mathcal{A}\mu_t)^x(\tilde{u})~\txtd\tilde{\omega}\\0\end{pmatrix}. 
\end{equation*}
Since the added component is simply $0$, the characteristic field $V$ again satisfies the regularity properties of Lemma \ref{lem: regularity of V}, which implies that we  obtain a regular flow $T_{t,t_0}[\mathcal{A},\mu,x] : \hat{G} \to \hat{G}$, and we can thus repeat the proof of existence and uniqueness of solutions for the fixed point equation (\ref{eq: fixed point eq assoc with the VFPE}) from Theorem \ref{thm: existence and uniqness of solutions for the fixed point eq}; see also \cite{Medv2016} and \cite[Section 4]{Medvedev2018}, where the additional frequencies effectively lead to one more integral but do not take any effect on the main argument of the proof for mean-field limits.

With our new view using graphops for large-scale dynamics on graphs, there are evidently now a lot of interesting open questions and challenges. We mention just a few open problem directly connected to our setting here:
\begin{itemize}
	\item Generalize our analysis, with the underlying assumption that the limiting graphop $A$ has bounded $(1,q)$ norm, to cover general $(p,q)$ bounded graphops. 
	\item Another limitation of our analysis is the assumption that the node space $\Omega$ is a CA group with the Haar measure, which means that the edges are distributed somewhat uniformly on the nodes. Hence, graphs with very inhomogeneous distribution of edges, like for instance star graphs with a ``giant'' node connected to every other node, are excluded from our analysis. One main goal for future research will be to allow such inhomogeneous  structures by considering the general case of any compact Borel probability space $(\Omega,\Sigma,m)$.
	\item One may also expect that further variations of the Kuramoto model on complex networks, e.g., the Kuramoto model involving second-order time derivatives~\cite{KuehnThrom2}, also have mean-field limits leading again to VFPEs on graphops.
	\item Even more generally, instead of Kuramoto-type models one could consider a completely abstract kinetic model of the form
	\begin{equation}
	\label{eq: kinetic model on graphs}
	\dot{u}_i = \sum_{j=1}^NA^N_{i,j}f(u_j, u_i), \quad i \in [N]:= \{1,2,...,N\},
	\end{equation}
	defined on finite networks which converge, as $N\to \infty$, towards a limiting graphop $A$. It is known that for the corresponding kinetic problem 
	\begin{equation}
	\dot{u}_i = \sum_{j=1}^Nf(u_j, u_i), \quad i \in [N]:= \{1,2,...,N\},
	\end{equation}
	the mean field VFPE reads as 
	\begin{equation}
	\partial{}_t \rho = - \partial{}_u\Big( \rho V[f](\rho)\Big),
	\end{equation}
	where the characteristic field $V[f]$ can be entirely computed from $f$ \cite{Golse2016}. As conjectured in~\cite{Kuehn2020}, the mean field VFPE for the kinetic model (\ref{eq: kinetic model on graphs}) should be given by
	\begin{equation}
	\partial{}_t \rho = - \partial{}_u\Big( \rho V[f](A\rho)\Big).
	\end{equation}
\end{itemize}

\medskip

\textbf{Acknowledgements:} MAG and CK gratefully thank the TUM International Graduate School of Science and Engineering (IGSSE) for support via the project ``Synchronization in Co-Evolutionary Network Dynamics (SEND)''. CK also acknowledges partial support by a Lichtenberg Professorship funded by the VolkswagenStiftung. MAG also thanks Xu Chuang for very helpful discussions and comments after reading the paper.

\bibliographystyle{plain}
\bibliography{bib1}

%\appendix	
\section*{Appendix: Some technical results}
\label{sec: appendix, technical lemmas}

\begin{lem}\label{lem: db and db2 are equival metrics}
	$\bar{d}^{b}$ is a metric on $\bar{\mathcal{M}}^b$.
\end{lem}

\begin{proof}
	We first note that $\bar{d}^b$ takes positive finite values since for any probability measure on $(\Omega,\Sigma)$, any $B \in \mathcal{G}^{l}$ and for $l$-a.a. $x \in \Omega$ we have ($d_{TV}$ denotes the total variation of measures):
	\begin{align*}
	0 \leq \int_\Omega d_{BL}(\mu^y,\kappa^y) ~\txtd\nu_x^B(y) &\leq \textnormal{diam}(\Omega) \int_\Omega \underbrace{d_{TV}(\mu^y,\kappa^y) }_{\leq b}~\txtd\nu_x^B(y)\\
	&\leq \textnormal{diam}(\Omega) b  \nu_x^B(\Omega) \\
	&\leq b\cdot \textnormal{diam}(\Omega) \gamma_B \\
	&\leq b\cdot \textnormal{diam}(\Omega).
	\end{align*}
	Note that for the graphop $A$ and any Borel measurable functions $f,g \in  L^\infty(\Omega,m)$ we have (due to the definition of the fiber measures and Theorem \ref{thm: measure repres of graphops}) 
	\begin{align*}
	\int_{\Omega^2} f(x)g(y) ~\txtd\nu_x^A(y) ~\txtd m(x)=\int_{\Omega^2}f(x)g(y)~\txtd\nu(x,y)  = 	\int_{\Omega^2} g(x)f(y) ~\txtd\nu_x^A(y) ~\txtd m(x)
	\end{align*}
	Thus, for two elements $\bar{\mu},\bar{\kappa} \in \bar{\mathcal{M}}^{b}$ we have following implications:
	\begin{align*}
	sup_{B \in \mathcal{G}^{\psi}}\sup_{x \in \Omega} \int_\Omega &d_{BL}(\mu^y,\kappa^y) ~\txtd\nu_x^B(y)  = 0 \\
	&\Rightarrow \int_\Omega d_{BL}(\mu^y,\kappa^y)~\txtd\nu_x^B(y)= 0  \text{ for }l-a.e. x \in \Omega \quad \forall B\in \mathcal{G}^{l} \quad \forall \text{ probability measures } l.\\
	&\Rightarrow \underbrace{\int_\Omega\int_\Omega d_{BL}(\mu^y,\kappa^y) ~\txtd\nu_x^B(y) dl(x)}_{ = \int_\Omega\int_\Omega d_{BL}(\mu^x,\kappa^x) ~\txtd\nu_x^B(y)dl(x) }= 0 \quad \forall B\in \mathcal{G}^{l} \quad \forall \text{ probability measures } l. \\
	&\Rightarrow \int_\Omega \nu_x^B(\Omega) d_{BL}(\mu^x,\kappa^x)dl(x) = 0  \quad \forall B\in \mathcal{G}^{l} \quad \forall \text{ probability measures } l.\\
	&\Rightarrow  \nu_x^B(\Omega) d_{BL}(\mu^x,\kappa^x) = 0 \quad l\text{-a.e. } x\in \Omega \quad \forall B\in \mathcal{G}^{l} \quad \forall \text{ probability measures } l.
	\end{align*}
	Now, by the observation that the set $\mathcal{G}^l$ surely contains graphops for which $\nu_x^B(\Omega) \neq 0$ for $l$-a.e.$x \in \Omega$ holds (we can consider for example the special case of graphons with $~\txtd\nu_x^B(y):= W(x,y)~\txtd l(y)$ and $W$ can be an arbitrary kernel) and that the measure $l$ can be concentrated at any point $x \in \Omega$ (we can consider for example the case that $l = \delta_x$ is as Dirac measure) we conclude that 
	\begin{equation*}
	\mu^y= \kappa^y \text{ for all } y\in \Omega.
	\end{equation*}
	The other properties of the metric are easy to check.
\end{proof}

\begin{lem} \textbf{(Completeness of $\bar{\mathcal{M}}^{b}$)} \label{lem: Mb is complete}\\
	The metric space $(\bar{\mathcal{M}}^{b},\bar{d}^{b})$ is complete.
\end{lem}
\begin{proof}
	The proof is adapted from  \cite[Lemma 2.1]{Medvedev2018}. Let $\{ \bar{\mu}_n \}_{n\in \mathbb{N}}$ be a Cauchy sequence in $\bar{\mathcal{M}}^{b}$. We have to show that this sequence converges in $\bar{\mathcal{M}}^{b}$. Since $\{ \bar{\mu_n} \}$ is a Cauchy sequence, there is an increasing sequence of indices $n_k$ such that 
	\begin{equation*}
	\bar{d}^{b}(\bar{\mu}_{n_k}, \bar{\mu}_{n_{k+1}}) =\sup_{B \in \mathcal{G}} \sup_{x\in \Omega}  \int_\Omega d_{BL}(\mu_{n_k}^y,\mu_{n_{k+1}}^y) ~\txtd\nu_x^B(y) < \frac{1}{2^{k+1}}, \quad k=1,2,... .
	\end{equation*}
	Using the monotone convergence theorem, this implies that for for any probability measure $l$ on $(\Omega, \Sigma)$, any $B \in \mathcal{G}^{l}$, and for $l$-a.a. $x\in \Omega$ we have that
	\begin{equation*}
	\int_\Omega \sum_{k=1}^\infty d_{BL}(\mu_{n_k}^y,\mu_{n_{k+1}}^y) ~\txtd\nu_x^B(y) = 
	\sum_{k=1}^\infty \int_\Omega d_{BL}(\mu_{n_k}^y,\mu_{n_{k+1}}^y) ~\txtd\nu_x^B(y)   < \infty
	\end{equation*}
	or in other words, the function $f: \Omega \to \mathbb{R}$ given $f(y)= \sum_{k=1}^\infty d_{BL}(\mu_{n_k}^y,\mu_{n_{k+1}}^y)$ is $\nu_x^B-$integrable for all $x\in \Omega$. This implies especially that 
	\begin{equation*}
	\sum_{k=1}^\infty d_{BL}(\mu_{n_k}^y,\mu_{n_{k+1}}^y)  < \infty \text{\quad for $\nu_x^B$-a.a. $y \in  \Omega$, for every $x \in \Omega$.}
	\end{equation*}
	Since for every indices $i,j$ with $j > i$ we have
	\begin{equation*}
	d_{BL}(\mu_{n_i}^y, \mu_{n_{j}}^y) \leq \sum_{k=i}^{j-1} d_{BL}(\mu_{n_k}^y, \mu_{n_{k+1}}^y) \to 0 \text{\quad as } i,j\to 0,
	\end{equation*}
	the sequence $\{\mu_{n_k}^y\}$ is Cauchy for $\nu_x^B$-a.a. $y \in  \Omega$, for every $x \in \Omega$.
	Hence, since the metric space $(\mathcal{M}_f,d_{BL})$ is complete, there exists the limit 
	\begin{equation*}
	\mu^y = \lim_{k\to \infty} \mu_{n_k}^y, \text{\quad for $\nu_x^B$-a.a. $y \in  \Omega$, for every $x \in \Omega$}
	\end{equation*}
	which is a measurable function as a limit of measurable functions.	
	Note further that for $\nu_x^B$-a.a. $y \in  \Omega$, for every $x \in \Omega$ we have that
	\begin{equation*}
	|\mu_{n_k}^y (\mathbb{T}) -\mu^y(\mathbb{T}) | \leq d_{BL}(\mu_{n_k}^y,\mu^y) \to 0 \quad \text{as } k\to \infty
	\end{equation*}
	which implies, due to the fact that $\bar{\mu}_{n_k} \in \bar{\mathcal{M}}^{b}$ that 
	\begin{equation*}
	\mu^y(\mathbb{T}) \leq b \quad \text{ \quad for $\nu_x^B$-a.a. $y \in  \Omega$, for every $x \in \Omega$}.
	\end{equation*}
	Since $B\in \mathcal{G}$ can be any graphop, w.r.t. an arbitary probability measure $l$, this implies (similarly to the proof of Lemma \ref{lem: db and db2 are equival metrics}) that for the limit $\mu$,  $\bar{\mu}\in \bar{\mathcal{M}}$ holds. To show that $\bar{\mu}$ is also the limit of the whole sequence, we note that for every indices $j> i$ we have
	\begin{align*}
	\bar{d}^{b}(\bar{\mu}_{n_i}, \bar{\mu}_{n_{j}}) &= \sup_{B \in \mathcal{G}}\sup_{x \in \Omega} \int_\Omega d_{BL}(\mu_{n_i}^y,\mu_{n_{j}}^y) ~\txtd\nu_x^B(y) \\
	&\leq \sup_{B \in \mathcal{G}}\sup_{x \in \Omega} \int_\Omega \sum_{k=i}^{j-1} d_{BL}(\mu_{n_k}^y,\mu_{n_{k+1}}^y) ~\txtd\nu_x^B(y) \\
	&\leq  \sum_{k=i}^{j-1}\sup_{B \in \mathcal{G}}\sup_{x \in \Omega}\int_\Omega  d_{BL}(\mu_{n_k}^y,\mu_{n_{k+1}}^y)~\txtd\nu_x^B(y) \\
	&\leq  \sum_{k=i}^{\infty} \underbrace{\sup_{B \in \mathcal{G}}\sup_{x \in \Omega}\int_\Omega  d_{BL}(\mu_{n_k}^y,\mu_{n_{k+1}}^y) ~\txtd\nu_x^B(y) }_{\leq \frac{1}{2^{k+1}}} \\
	&\leq \frac{1}{2^i} \to 0 \text{\quad as } i \to \infty.
	\end{align*}
	Now using the continuity of the metric and the dominated convergence theorem we obtain
	\begin{align*}
	\bar{d}^{b}(\bar{\mu}_{n_i}, \bar{\mu}) &=\sup_{B \in \mathcal{G}} \sup_{x \in \Omega} \int_\Omega d_{BL}(\mu_{n_i}^y,\mu^y) ~\txtd\nu_x^B(y)\\
	&= \sup_{B \in \mathcal{G}}\sup_{x \in \Omega} \int_\Omega \lim_{j \to \infty} d_{BL}(\mu_{n_i}^y,\mu_{n_j}^y) ~\txtd\nu_x^B(y) \\
	&= \sup_{B \in \mathcal{G}} \sup_{x \in \Omega} \lim_{j \to \infty}\int_\Omega  d_{BL}(\mu_{n_i}^y,\mu_{n_j}^y) ~\txtd\nu_x^B(y)\\
	&\leq  \limsup_{j \to \infty} \underbrace{\sup_{B \in \mathcal{G}}\sup_{x \in \Omega}\int_\Omega  d_{BL}(\mu_{n_i}^y,\mu_{n_j}^y) ~\txtd\nu_x^B(y)}_{\leq \frac{1}{2^i}}\\
	& \to 0 \text{\quad as } i \to \infty.
	\end{align*}
	Hence, the subsequence $\bar{\mu}_{n_i}$ converges to $\bar{\mu}$. Since it is a subsequence of the Cauchy sequence $\bar{\mu}_{n}$, this implies already the convergence of the whole sequence towards $\bar{\mu}$, i.e.
	\begin{equation*}
	\bar{d}^{b}(\bar{\mu}_{n}, \bar{\mu}) \to 0 \text{\quad as } n\to \infty.
	\end{equation*}
	Hence, the space $(\bar{\mathcal{M}}^{b},\bar{d}^{b})$ is complete.	
\end{proof}

\begin{lem} \cite[Lemma 2.5]{Medvedev2018} \textbf{(Gronwwall's lemma)} 
	\label{lem: Gronwalls lemma} \\
	Let $\phi(t)$ and $\alpha(t)$ be continuous functions on $[0,T]$ and
	\begin{equation*}
	\phi(t) \leq A \int_0^t \phi(s) ~\txtd s + B \int_0^t \alpha(s)~\txtd s + C, \quad t\in [0,T],
	\end{equation*}
	where $A\geq 0$. Then
	\begin{equation*}
	\phi(t) \leq \txte^{At}\Big( B \int_0^t \alpha(s) \txte^{-As} ~\txtd s + C\Big).
	\end{equation*}
\end{lem}

\begin{proof}[Proof of Theorem \ref{thm: existence and uniqness of solutions for the fixed point eq}]
	We follow the lines of the proof of \cite[Theorem 2.4]{Medvedev2018}. \\
	\textbf{(I)} First of all, it is easy to see that for any $\mu \in \mathcal{M}^{b}_{\mathcal{T}}$, $\overline{\mathcal{F}\mu_t} \in \bar{\mathcal{M}}^b$ holds. Further, for any times $t_0, t\in \mathcal{T}$, w.l.o.g. $t\geq t_0$,  we calculate, using a change of variables (in the third equation),
	\begin{align}\label{eq: thm: A is a constaction, 3}
	\bar{d}^{b} ( \mathcal{F}[\mu](t,\cdot), \mathcal{F}[\mu](t_0,\cdot)) &= \bar{d}^{b} (\bar{\mu}_0\circ T_{0,t}[\mu,\cdot],\bar{\mu}_0\circ T_{0,t_0}[\mu,\cdot]) \nonumber \\
	&= \sup_{B\in \mathcal{G}}\sup_{x \in \Omega} \int_\Omega d_{BL} (\mu_0^y\circ T_{0,t}^y[\mu],\mu_0^y\circ T_{0,t_0}^y[\mu]) ~\txtd\nu_x^B(y) \nonumber \\
	&=\sup_{B\in \mathcal{G}}\sup_{x \in \Omega}\int_\Omega \sup_{f\in \mathcal{L}} \Big| \int_G  \Big( f(T^y_{t,0}[\mu]v)- f(T^y_{t_0,0}[\kappa]v) \Big) ~\txtd\mu_0^y(v) \Big|  ~\txtd\nu_x^B(y) \nonumber\\
	&\leq \sup_{B\in \mathcal{G}}\sup_{x\in \Omega}\int_\Omega \int_\mathbb{T} | T^y_{t,0}[\mu]v - T^y_{t_0,0}[\mu]v |~\txtd\mu_0^y(v)   ~\txtd\nu_x^B(y) \nonumber\\
	&\leq \int_{t_0}^t \sup_{B\in \mathcal{G}}\sup_{x\in \Omega}\int_\Omega \int_\mathbb{T} \underbrace{|V[\mathcal{A},\mu, x](s,T^y_{s,0}[\mu]v) |}_{\leq C \parallel D \parallel_\infty b \gamma_A \text{ by Lemma \ref{lem: V, T are uniformly bounded} }} ~\txtd\mu_0^y(v)   ~\txtd\nu_x^B(y) ~\txtd s \nonumber \\
	&\leq C\parallel D \parallel_\infty b^2  (t-t_0) \to 0 \text{ as $t \to t_0$}.
	\end{align}	
	This shows that $\mathcal{F}\mu \in \mathcal{M}^{b}_{\mathcal{T}}$. Thus, $\mathcal{F}$ is well-defined.
	Now let $\mu, \kappa \in \mathcal{M}^{b}_{\mathcal{T}}$.
	As before we calculate that 
	\begin{align}\label{eq: thm: A is a constaction, 1}
	\bar{d}^{b} ( \mathcal{F}[\mu](t,\cdot), \mathcal{F}[\kappa](t,\cdot)) &= \bar{d}^{b} (\bar{\mu}_0\circ T_{0,t}[\mu,\cdot],\bar{\mu}_0\circ T_{0,t}[\kappa,\cdot]) \nonumber \\
	&= \sup_{B\in \mathcal{G}}\sup_{x \in \Omega} \int_\Omega d_{BL} (\mu_0^y\circ T_{0,t}^y[\mu],\mu_0^y\circ T_{0,t}^y[\kappa]) ~\txtd\nu_x^B(y) \nonumber \\
	&=\sup_{B\in \mathcal{G}}\sup_{x \in \Omega}\int_\Omega   \sup_{f\in \mathcal{L}}\Big|\int_G  \Big( f(T^y_{t,0}[\mu]v)- f(T^y_{t,0}[\kappa]v) \Big) ~\txtd\mu_0^y(v) \Big|  ~\txtd\nu_x^B(y) \nonumber\\
	&\leq \sup_{B\in \mathcal{G}}\sup_{x \in \Omega}\int_\Omega \int_\mathbb{T} | T^y_{t,0}[\mu]v - T^y_{t,0}[\kappa]v |~\txtd\mu_0^y(v) ~\txtd\nu_x^B(y) =: \lambda(t)
	\end{align}	
	Here we used again a change of variables. Using the triangular inequality we calculate
	\begin{align} \label{eq: thm: A is a constaction, 2}
	\lambda(t) &=\sup_{B\in \mathcal{G}}\sup_{x \in \Omega} \int_\Omega\int_{\mathbb{T}} | T^y_{t,0}[\mu]v - T^y_{t,0}[\kappa]v |~\txtd\mu_0^y(v)  ~\txtd\nu_x^B(y) \nonumber\\
	&\leq \sup_{B\in \mathcal{G}}\sup_{x \in \Omega}\int_0^t \int_\Omega\int_\mathbb{T} |V[\mu,y](s,T^y_{s,0}[\mu]v) - V[\kappa,y](s,T^y_{s,0}[\kappa]v) |~\txtd\mu_0^y(v)  d \nu_x^B(y) ~\txtd s \nonumber\\
	&\leq\sup_{B\in \mathcal{G}}\sup_{x \in \Omega}  \int_0^t\int_\Omega \int_\mathbb{T} |V[\mu,y](s,T^y_{s,0}[\mu]v) - V[\kappa,y](s,T^y_{s,0}[\mu]v) |~\txtd\mu_0^y(v)~\txtd\nu_x^B(y)    ~\txtd s \nonumber\\
	&+\sup_{B\in \mathcal{G}}\sup_{x \in \Omega} \int_0^t \int_\Omega\int_\mathbb{T}|V[\kappa,y](s,T^y_{s,0}[\mu]v) - V[\kappa,y](s,T^y_{s,0}[\kappa]v) |~\txtd\mu_0^y(v)  ~\txtd\nu_x^B(y) ~\txtd s.
	\end{align}
	With this notation we calculate for the first difference, using Lemma \ref{lem: regularity of V} 
	\begin{align*}
	\sup_{B\in \mathcal{G}}\sup_{x \in \Omega} \int_0^t \int_\Omega\int_\mathbb{T}  &|V[\mu,y](s,T^y_{s,0}[\mu]v) - V[\kappa,y](s,T^y_{s,0}[\mu]v) |~\txtd\mu_0^y(v)  ~\txtd\nu_x^B(y)  ~\txtd s  \\
	&\leq 2C\int_0^t \underbrace{\sup_{B\in \mathcal{G}} \sup_{x\in \Omega}\Big( \nu_x^B(\Omega) \Big)}_{\leq \sup_{B\in \mathcal{G}} \gamma_B \leq 1}\underbrace{\mu_0(\mathbb{T})}_{\leq b}    \bar{d}^{b} (\bar{\mu}_s,\bar{\kappa}_s )    ~\txtd s\\
	&\leq   2Cb \int_0^t\bar{d}^{b}(\bar{\mu}_s,\bar{\kappa}_s)~\txtd s
	\end{align*}
	For the second difference we calculate, again using Lemma \ref{lem: regularity of V}
	\begin{align*}
	\sup_{B\in \mathcal{G}}\sup_{x \in \Omega} \int_0^t \int_\Omega\int_\mathbb{T} &|V[\kappa,y](s,T^y_{s,0}[\mu]v) - V[\kappa,y](s,T^y_{s,0}[\kappa]v) |~\txtd\mu_0^y(v)  ~\txtd\nu_x^B(y) ~\txtd s \\
	&\leq  b\gamma_A  \int_0^t \sup_{B\in \mathcal{G}}\sup_{x \in \Omega} \int_\Omega\int_\mathbb{T}|T^y_{s,0}[\mu]v - T^y_{s,0}[\kappa]v|~\txtd\mu_0^y(v)  ~\txtd\nu_x^B(y) ~\txtd s \\
	&= b \gamma_A \int_0^t \lambda(s) ~\txtd s.
	\end{align*}
	We set $C_1:=2Cb$ and  $C_2:=b \gamma_A$. 
	Substituting both expressions in (\ref{eq: thm: A is a constaction, 2}) we get
	\begin{align*}
	\lambda(t) \leq C_1\int_0^t\bar{d}^{b}(\bar{\mu}_s,\bar{\nu}_s)~\txtd s + C_2 \int_0^t \lambda(s) ~\txtd s.
	\end{align*}
	Using Gronwall's inequality, (cf. Lemma \ref{lem: Gronwalls lemma}) we obtain
	\begin{equation}
	\lambda(t) \leq C_1 \txte^{C_2t} \int_0^t\bar{d}^{b}(\bar{\mu}_s,\bar{\nu}_s) \txte^{-C_2s}~\txtd s.
	\end{equation}
	Using (\ref{eq: thm: A is a constaction, 1}) this implies that 
	\begin{align*}
	\bar{d}^{b} ( \mathcal{F}[\mu](t,\cdot), \mathcal{F}[\kappa](t,\cdot)) &\leq C_1 \txte^{C_2t} \int_0^t\bar{d}^{b}(\bar{\mu}_s,\bar{\nu}_s) \txte^{-C_2s}~\txtd s.
	\end{align*}
	Hence, 
	\begin{align*}
	d_\alpha^{b} ( \mathcal{F}[\mu](t,\cdot), \mathcal{F}[\kappa](t,\cdot) &= \sup_{t \in \mathcal{T}} \Big\{ \txte^{-\alpha t}\bar{d}^{b} (\mathcal{F} [\mu](t,\cdot), \mathcal{F}[\kappa](t,\cdot)) \Big\} \\
	&\leq \sup_{t \in \mathcal{T}} 
	C_1 \txte^{-(\alpha-C_2)t}\int_0^t\bar{d}^{b}(\bar{\mu}_s,\bar{\kappa}_s) \txte^{-C_2s}~\txtd s \\
	&\leq C_1 d_\alpha^{b}(\bar{\mu}, \bar{\kappa}) \sup_{t \in \mathcal{T}} \txte^{-(\alpha-C_2)t}\int_0^t \txte^{(\alpha-C_2)s} ~\txtd s\\
	\leq C_1 (\alpha-C_2)^{-1}d^{b}_\alpha(\bar{\mu}, \bar{\kappa}).
	\end{align*}
	This proves the claim.\\
	 \textbf{\textbf{(II)}} This follows immediately from \textbf{(I)} and the Banach contraction principle in the complete metric space $\mathcal{M}^{b}_{\mathcal{T}}$.
\end{proof}

\begin{lem} \textbf{(Completeness of $\bar{\mathcal{M}}^{b,2}$)} \label{lem: Mb,c is complete}\\
	The metric space $(\bar{\mathcal{M}}^{b,2},\bar{d}^{b})$ is complete.
\end{lem}
\begin{proof}
	Let $\{ \bar{\mu}_n \}_{n\in \mathbb{N}}$ be a Cauchy sequence in $\bar{\mathcal{M}}^{b,2}$. Similarly to Lemma \ref{lem: Mb,c is complete} one can show $\{\mu_{n_k}^y\}$ is Cauchy for $\nu_x^B$-a.a. $y \in  \Omega$, for every $x \in \Omega$ there exists the limit $\mu \in \mathcal{M}^b$,
	\begin{equation*}
	\mu^y = \lim_{k\to \infty} \mu_{n_k}^y, \text{\quad for $\nu_x^B$-a.a. $y \in  \Omega$, for every $x \in \Omega$}.
	\end{equation*}
	By the triangular inequality we see that for any $S \in \mathcal{B}(\mathbb{T})$ with $\lambda(\partial{S}) = 0$ and $x,y\in \Omega$ we have
	\begin{equation*}
	|\mu^x(S) -\mu^{y}(S)|  \leq |\mu^x(S) -\mu^{x}_{n_k}(S) | + |\mu^{x}_{n_k}(S) - \mu^y_{n_k}(S)| + |\mu^{y}_{n_k}(S) - \mu^y(S)|.
	\end{equation*}
	Since the right side can be made arbitary small as $k\to \infty$ and $x\to y$, the map 
	$x\mapsto \mu^x(S)$ is continuous. Thus, $\mu \in \mathcal{M}^{b,2}$ and  $(\bar{\mathcal{M}}^{b,2},\bar{d}^{b})$ is a closed subspace of the complete metric space $(\bar{\mathcal{M}}^{b},\bar{d}^{b}) $. This finished the proof.	
\end{proof}
%%%%%%%%%%%%%%%%%%%%%%%%%%%%%%%%%%%%%%%%%%%%%%%%%%%%%%%%%%%%%%%%%%%%%%%%%%%%%%%%%%%%%%%%%%%%%%% end of paper 1%%%%%%%%%%%%%%%

\end{document}